\documentclass{article}

\usepackage{amsmath,amsthm,amssymb}
\usepackage{changes}
\usepackage{microtype}
\usepackage{enumerate}
\usepackage{graphicx}
\usepackage{subfigure}
\usepackage{booktabs} 
\usepackage{color}		

\newtheorem{assumption}{Assumption}
\newtheorem{theorem}{Theorem}
\newtheorem{definition}{Definition}
\newtheorem{lemma}[theorem]{Lemma}

\newtheorem{rem}{Remark}
\newtheorem{example}{Example}

\newtheorem{proposition}{Proposition}

\usepackage{hyperref}

\newcommand{\mb}[1]{\mathbf{#1}}
\newcommand{\mbb}[1]{\mathbb{#1}}
\newcommand{\mc}[1]{\mathcal{#1}}

\def\T{\operatorname{T}}
\def\Id{\operatorname{Id}}
\def\R{\mathbb{R}}
\def\N{\mathbb{N}}
\def\prox{\operatorname{prox}}
\def\hbx{\hat{\mathbf{x}}}
\def\hx{\hat{x}}
\def\bx{\mathbf{x}}
\def\by{\mathbf{y}}
\def\bz{\mathbf{z}}

\usepackage[accepted]{icml2023}

\icmltitlerunning{Delay-agnostic Asynchronous Coordinate Update Algorithm}

\begin{document}

\twocolumn[
\icmltitle{Delay-agnostic Asynchronous Coordinate Update Algorithm}




\begin{icmlauthorlist}
		\icmlauthor{Xuyang Wu}{kth}
		\icmlauthor{Changxin Liu}{kth}
		\icmlauthor{Sindri Magn\'{u}sson}{su}
		\icmlauthor{Mikael Johansson}{kth}
	\end{icmlauthorlist}
	\icmlaffiliation{kth}{Division of Decision and Control Systems, EECS, KTH Royal Institute of Technology, Stockholm, Sweden}
	\icmlaffiliation{su}{Department of Computer and System Science, Stockholm University, Stockholm, Sweden}
	\icmlcorrespondingauthor{Xuyang Wu}{xuyangw@kth.se}

	\icmlkeywords{Machine Learning, ICML}

\vskip 0.3in
]



\printAffiliationsAndNotice{}  

\begin{abstract}

We propose a delay-agnostic asynchronous coordinate update algorithm (DEGAS) for computing operator fixed points, with applications to asynchronous optimization. DEGAS includes novel asynchronous variants of ADMM and block-coordinate descent as special cases. We prove that DEGAS converges under both bounded and unbounded delays under delay-free parameter conditions. We also validate by theory and experiments that DEGAS adapts well to the actual delays. The effectiveness of DEGAS is demonstrated by numerical experiments on classification problems.

\end{abstract}

\section{Introduction}\label{sec:intro}

Many popular algorithms in machine learning, optimization, and game theory can be formulated as fixed point iterations
\begin{equation}\label{eq:fpupdate}
    \mathbf{x}(k+1) = \operatorname{T}(\mb{x}(k)),
\end{equation}
where $k$ is the iteration index, $\bx(k)\in\mathbb{R}^d$ is the iterate at iteration $k$, and $\T:\mathbb{R}^d\rightarrow \mathbb{R}^d$ is an operator. For example, the gradient descent method for minimizing a differentiable function $f$ is on the form \eqref{eq:fpupdate} with $\T(\bx)=\bx-\gamma \nabla f(\bx)$ for some positive step-size parameter $\gamma>0$.


In machine learning applications, the problem dimension is sometimes so large that evaluating the full operator $\T$ in each iteration is impractical.
For these problems, coordinate update methods \cite{Nesterov12, wright2015coordinate} have proven to be very competitive. 
These methods split the decision vector $\bx$ into multiple blocks, $\bx=(x_1,\ldots,x_m)$, and only update one block $i$ in each iteration
\begin{equation}\label{eq:coordinateupdate}
    x_i(k+1) = \T_i(\mb{x}(k)),
\end{equation}
while $x_{j}(k+1)=x_j(k)$ for all $j\neq i$.
Here, $\T_i$ is the $i$th block of $\T$ such that $\T(\bx)=(\T_1(\bx), \ldots, \T_m(\bx))$. In many cases, the cost of computing $\T_i$ can be much lower than that of computing the whole $\T$~\cite{Nesterov12}.



A natural approach for accelerating coordinate update methods is to implement them on multiple processors/machines in a distributed environment. For example, in each iteration we may let every processor compute $\T_i$ for a randomly selected block $i$, update all the selected blocks, and then go to the next iteration \cite{richtarik2016distributed}. We consider this to be a synchronous update, since all the processors are synchronized and the algorithm does not proceed to the next iteration until all processors finish their work. Due to the use of multiple processors, synchronous coordinate update methods can converge significantly faster than the centralized coordinate update \eqref{eq:coordinateupdate}. However, their convergence speed is bottlenecked by the slowest processor and they are sensitive to single-node failures. In contrast, asynchronous coordinate update methods eliminate the need for global synchronization and can be more efficient and robust.

This paper focuses on asynchronous coordinate updates. Although these are useful in a wide range of sciences,
this work only discusses applications  in optimization and ML.


\subsection{Related work}

In the past few decades, there has been a growing interest in developing parallel and asynchronous machine learning algorithms. As a part of this effort, a large number of asynchronous and distributed optimization algorithms with strong practical performance have been developed, including
Async-SGD \cite{recht2011hogwild}, Asynchronous ADMM \cite{zhang2014asynchronous}, PIAG \cite{aytekin16, Sun19, Feyzmahdavian21}, Async-BCD \cite{liu2014,Wu2022adaptive}, DAve-RPG \cite{mishchenko2018delay}, DAve-QN \cite{soori20a}, and ADSAGA \cite{glasgow2022asynchronous}. Most of these algorithms are tailored to specific computing architectures such as master-worker \cite{LiM13} or shared-memory \cite{bertsekas2003parallel}, while algorithms such as AsySPA \cite{zhang2019asyspa}, DFAL \cite{aybat2015asynchronous}, and the Asynchronous primal-dual algorithm \cite{wu2017decentralized}
consider general communication topologies. 

Two of the most
influential frameworks for asynchronous coordinate update methods 
are due to 
\cite{bertsekas1983distributed} and \cite{Peng16}, respectively. In contrast to the related works cited above, which focus on solving specific classes of optimization problems, they consider asynchronous coordinate updates for the more general problem of finding fixed points of operators. Specifically, \citet{bertsekas1983distributed} proposes the following asynchronous implementation of~\eqref{eq:coordinateupdate}:
\begin{equation}\label{eq:Bertsekas}
    x_i(k+1) = \T_i(\hbx(k)),
\end{equation}
where $\hbx(k) = (\hx_1(k), \ldots, \hx_n(k))$ with each $\hx_j(k)=x_j(k-\tau_j(k))$ for some integer $\tau_j(k)\ge 0$. Here,  $\tau_j(k)$ represents the information delay from node $j$, i.e. the difference between the current iteration index and the index of the iterate block used for computing $\T_i$. However, this framework rarely applies to machine learning problems, since it is only guaranteed to converge if $\T$ is contractive in a block-maximum norm; see~\S~2.1.1. Even for gradient descent iterations on quadratic optimization problems, this condition only holds if the Hessian is diagonally dominant.  


The ARock framwork of~\citet{Peng16} considers the modified coordinate updates
\begin{equation}\label{eq:ARock}
    x_i(k+1) = x_i(k)+\gamma(k)(\T_i(\hbx(k))-\hx_i(k)),
\end{equation}
where $\gamma(k)>0$ is the step-size. Unlike \cite{bertsekas1983distributed}, ARock only requires $\T$ to be non-expansive and applies to 
modern algorithms like BCD \cite{Nesterov12} and ADMM \cite{boyd2011distributed}.
However, like most asynchronous algorithms that use fixed step-sizes, such as PIAG \cite{aytekin16} and Async-BCD \cite{liu2014}, existing convergence results require that delays are uniformly bounded and rely on step-size restrictions for $\gamma(k)$ that depend on this (typically unknown) delay bound. This causes difficulties in practice: using a large delay bound (to ensure that it is valid) leads to a small step-size, and unnecessarily slow convergence. In addition, guarding against the maximum delay leads to overly conservative results if most delays are smaller than the maximum delay.  Indeed, a number of recent papers report delay measurements for asynchronous optimization algorithms that show that real-world delays tend to be distributed in this way; see, \emph{e.g.},  \cite{mishchenko2022asynchronous, Wu2022adaptive, koloskovasharper} and our own measurements in Figure \ref{fig:delayLASSOCIFAR} in Appendix \ref{sec:delaydistribution}. As a specific example, 
\citet{mishchenko2022asynchronous} implement an asynchronous SGD on a 40-core CPU and report a maximum and average delay of around 1200 and 20, respectively.

\subsection{Contribution}


In this paper, we propose an alternative way to perform asynchronous coordinate updates. This approach, which we call the 
DElay-aGnostic ASynchronous coordinate update (DEGAS) algorithm, 
adapts the updates 
\eqref{eq:Bertsekas} to a master-worker architecture \cite{LiM13} and samples the update block uniformly at random. We show that with these modifications, the new algorithm preserves the advantages of \eqref{eq:Bertsekas} and \eqref{eq:ARock} and avoids their drawbacks in the sense that

i) Like \eqref{eq:Bertsekas}, DEGAS is free from parameters that depend on the delay. In this way, it avoids ARock's issues with hard-to-determine and conservative step-sizes. Moreover, by characterizing how the convergence of DEGAS is affected by the distribution of delays in a stochastic delay model, 
we show that convergence is faster when small delays are more likely than large delays. This is in contrast to  ARock, whose convergence rate is dominated by the worst-case (largest) delay and whose performance does not improve even if the actual delays are much smaller. This can be observed by scrutinising the convergence bounds in \cite{Peng16} and is confirmed 
in our numerical results. 

ii) DEGAS converges under the same conditions on $\T$ as ARock, and can therefore be used for parallel and asynchronous implementations of 
a wide range of 
modern optimization methods, including 
BCD and ADMM. We prove that DEGAS converges under both bounded and unbounded delays. For bounded delays, we provide an explicit convergence rate and show that the iterates of DEGAS converge faster than the best-known bound for ARock.
To derive this result, we prove a linear rate for a general class of asynchronous sequences  which significantly sharpens a lemma from \cite{Feyzmahdavian21}. 

We illustrate the superior performance of DEGAS (including ADMM and BCD) in training of large scale models.

\subsection*{Notation and Preliminaries}

We let $\mathbb{N}$ be the set of natural numbers, and $\mathbb{N}_0=\N\cup\{0\}$. We denote $[m] = \{1,\ldots,m\}$ for any $m\in\mathbb{N}$ and define the proximal operator of a function $R:\mathbb{R}^d\rightarrow\mathbb{R}\cup\{+\infty\}$ as
$\operatorname{prox}_{R}(x) = \operatorname{\arg\!\min}_{y\in\mbb{R}^d} ~R(y)+\frac{1}{2}\|y-x\|^2$. We call a differentiable function $f:\R^d\rightarrow\R$  $L$-smooth if $\langle \nabla f(\bx)-\nabla f(\by), \by-\bx\rangle \le L\|\by-\bx\|^2~\forall \bx,\by\in\mathbb{R}^d$, and $\mu$-strongly convex if $\langle \nabla f(\bx)-\nabla f(\by), \by-\bx\rangle \ge \mu\|\by-\bx\|^2~\forall \bx,\by\in\mathbb{R}^d$. We use $\Id$ to denote the identity operator of proper dimension. For any operator $\T:\R^d\rightarrow\R^d$, $\mbox{Fix} \T=\{\bx: \bx=\T\bx\}$ represents its set of fixed-points. We use $\|\cdot\|$ to represent the Euclidean norm for vectors and the spectral norm for matrices. For any vector $\bx=(x_1,\ldots,x_m)\in\mathbb{R}^N$ and $w=(w_1,\ldots,w_m)\in\mathbb{R}^m$ where each $x_i\in\mathbb{R}^{d_i}$ and $w_i>0$, we define $\|\bx\|_{b,\infty}^w=\max_{i\in [m]} \frac{\|x_i\|_i}{w_i}$ as the block-maximum norm, where each $\|\cdot\|_i$ can be any vector norm.

\vspace{-0.1cm}

\section{Algorithm and main result}


In this section, we present our algorithm for finding the fixed point of an operator $\T$, analyze its convergence, and highlight its advantages over ARock \cite{Peng16}.

\subsection{Algorithm}

We adapt the asynchronous update \eqref{eq:Bertsekas} to the widely-used master-worker architecture \cite{LiM13} for distributed learning. Here, a master node stores the current model $\bx$ and coordinates the work of $n$ compute nodes. Each worker $w\in[n]$ asynchronously and continuously receives $\bx$ from the master, stores it in the local variable $\bx^w$, computes $\T_i(\bx^w)$ for some $i\in[m]$ drawn uniformly at random, and returns $\T_i(\bx^w)$ to the master. Once the master receives $\T_i(\bx^w)$ from some worker $w$, it updates
\begin{equation}\label{eq:asyncupdate}
    x_i = \T_i(\bx^w)
\end{equation}
and pushes the updated model back to the idle workers. 
A detailed implementation is given in Algorithm \ref{alg:Coor}, which we refer to as the DElay-aGnostic ASynchronous coordinate update (DEGAS) algorithm.
	\begin{algorithm}[tb]
		\caption{DEGAS}
		\label{alg:Coor}
		\begin{algorithmic}[1]
		    \STATE {\bfseries Setup:} initial iterate $\bx(0)$.
			\STATE {\bfseries Initialization:} the master sets $\bx=\bx(0)$ and broadcasts $\bx$ to all workers.
			\WHILE{\emph{not interrupted by master}: each worker $w\in [n]$ \emph{asynchronously} and \emph{continuously}}
			\STATE receive $\bx$ from the master and assign $\bx^w=\bx$.
			\STATE sample $i\in [m]$ uniformly at random.
			\STATE compute $\T_i(\bx^w)$.
			\STATE send $(\T_i(\bx^w),i)$ to the master.
			\ENDWHILE
			\WHILE{\emph{not converged}: the master}
			\STATE receive $(\T_i(\bx^w),i)$ from a worker $w$.
			\STATE update $x_i\leftarrow\T_i(\bx^w)$.
			\STATE send $\bx$ to the worker $w$.
			\ENDWHILE
		\end{algorithmic}
	\end{algorithm}

For the convenience of further discussion, we index the iterates by $k\in \N_0$, which represents the number of updates by the master, and use $i(k)$ to denote the updated block at time $k$. Note that in DEGAS, each $\bx^w$ in \eqref{eq:asyncupdate} is a delayed iterate and equals to $\bx(k-\tau(k))$ for some integer $\tau(k)\in [0,k]$. We refer to $\tau(k)$ as the delay at time $k$. In this way, the update at time $k\in\N_0$ can be equivalently rewritten as
\begin{equation}\label{eq:Alg1formal}
    x_j(k+1) = \begin{cases}
        \T_j(\bx(k-\tau(k))), & j=i(k),\\
        x_j(k), & \text{otherwise}.
    \end{cases}
\end{equation}

\subsubsection{Connection with existing works}\label{ssec:connection}

DEGAS can be viewed as an adaption of the asynchronous update \eqref{eq:Bertsekas} to the master-worker architecture. They are logically equivalent except for the particular block selection rule in DEGAS. Existing works on \eqref{eq:Bertsekas} mainly focused on the setting where the numbers of blocks and processors (or workers) are identical and each processor updates a certain block. Under this setting, to guarantee convergence~they often require $\T$ to be contractive in the block-maximum norm \cite{bertsekas2003parallel}, \emph{i.e.} to satisfy
\begin{equation}\label{eq:maxnormcontraction}
    \|\T(\bx)-\T(\bx^\star)\|_{b,\infty}^w<c\|\bx-\bx^\star\|_{b,\infty}^w
\end{equation}
for some $c\in(0,1)$, some $\bx^\star\in \mbox{\rm Fix} \T$ and for every $\bx\in\mathbb{R}^d$, where the block-maximum norm $\|\cdot\|_{b,\infty}^w$ is defined in Section \ref{sec:intro}. The condition \eqref{eq:maxnormcontraction} is restrictive and only holds for very specific operators, e.g., \cite{FROMMER1991105}, \cite{bertsekas2003parallel}, \cite{mehyar2007asynchronous}, \cite{moallemi2010convergence}, and \cite{hale2017asynchronous}. Even for the simple operator $\T=\Id-\frac{1}{L}\nabla f$ where $f(x)=\frac{1}{2}\bx^TA\bx+b^T\bx$ for a symmetric and positive definite matrix $A\in\mathbb{R}^{m\times m}$ and a vector $b\in\mathbb{R}^m$ and $L=\|A\|_2$, the condition \eqref{eq:maxnormcontraction} is known to hold only when $A$ is diagonally dominant. In the next subsection we will show that DEGAS, 
by letting each processor update a random block in \eqref{eq:Bertsekas}, can converge under a much weaker condition.

The ARock framework~\cite{Peng16} uses updates that are rather different. First, 
while~\eqref{eq:Bertsekas} only depends on $\hat{\bx}(k)$, the ARock updates \eqref{eq:ARock} are based on both $\hat{\bx}(k)$ and $x_i(k)$. To guarantee convergence, the maximally allowable step-size depends on the (usually unknown and large) worst-case delay, and decays quickly with the upper bound on the delays. This makes the algorithm difficult to tune and unnecessarily slow in practice. In contrast, DEGAS does not need access to the upper delay bound for tuning, but converges for all bounded delays. Moreover, Example 1 in \cite{feyzmahdavian2014delayed} provides a comparison between two delayed gradient methods, which are special cases of DEGAS and ARock with one block and one worker, respectively. They show that for a simple problem, the method specialized from DEGAS strictly outperforms the one from ARock, which suggests the superiority of the algorithmic form of DEGAS. We admit that the ARock framework is more flexible because it allows for inconsistent read and write while DEGAS does not. Due to this reason, ARock can be implemented on both the master-worker and the shared memory system, while DEGAS can only be implemented in the former where inconsistent read and write can be practically avoided.

Some existing asynchronous optimization methods can also converge with step-sizes that do not rely on the worst-case delay. Their step-sizes can be categorized as 1) delay-free fixed step-size; 2) delay-adaptive step-size; 3) delay-free diminishing step-size. We are only aware of four other asynchronous algorithms that converge with delay-free fixed step-sizes: the delayed proximal gradient method~\cite{feyzmahdavian2014delayed}, the asynchronous ADMM \cite{zhang2014asynchronous}, DAve-RPG \cite{mishchenko2018delay}, and the asynchronous level bundle method~\cite{iutzeler2020asynchronous}. The first three algorithms are different from DEGAS and, unfortunately, do not cover
coordinate update methods like BCD and ADMM. \citet{zhang2014asynchronous} assume that at each iteration, each worker has the same probability of sending results to the master, which is less practical. The works \cite{sra2016adadelay,Wu2022adaptive,cohen2021asynchronous,koloskovasharper} avoid using the worst-case delay by adapting step-sizes to the actual delays or the errors caused by actual delays, where \cite{Wu2022adaptive} studies PIAG and the asynchronous BCD and the remaining focus on the asynchronous SGD. The works \cite{agarwal2011distributed,zhou2018distributed,aviv21a} show convergence of the asynchronous SGD or its variants, under delay-free diminishing step-sizes that are effective in stochastic optimization but may lead to slow convergence if we apply them to deterministic optimization.

\subsection{Convergence analysis}\label{ssec:convana}

Throughout the paper, we assume the independence between the delays and the selected blocks.
\begin{assumption}\label{asm:blockdelayindependent}
    The delay sequence $\{\tau(k)\}_{k\in\N_0}$ and the block sequence $\{i(k)\}_{k\in\N_0}$ are independent.
\end{assumption}
Assumption \ref{asm:blockdelayindependent} is a standard assumption and is assumed in many asynchronous optimization works, e.g., ARock \cite{Peng16}, asynchronous SGD \cite{recht2011hogwild, mishchenko2022asynchronous}, and asynchronous coordinate descent \cite{liu2015}. However, it may not hold in practice if $\T_i$ is more expensive to compute for some block $i$ than the others \cite{leblond:18}. Recent advances for relaxing Assumption \ref{asm:blockdelayindependent} include before read labeling \cite{Mania17}, after read labeling \cite{leblond:18}, and single coordinate consistent ordering \cite{cheung2021fully}.

We first consider the case where all delays are bounded.
\begin{assumption}[Partial asynchrony]\label{asm:boundedelay}
	For some $\bar{\tau}\in\N_0$, $\tau(k)\le \bar{\tau}$ for all $k\in \N_0$.
\end{assumption}

We analyze two classes of operators $\T$ defined next. 
\begin{definition}[Averaged operator]\label{def:firmlynonex}
    The operator $\T:\R^d\rightarrow\R^d$ is an $\alpha$-averaged operator if 
    $\T=(1-\alpha)\Id + \alpha R$ for some $\alpha\in(0,1)$ and some non-expansive operator $R$. 
\end{definition}

\begin{definition}[Pseudo-contractive operator]\label{def:contraction}
    The operator $\T:\R^d\rightarrow\R^d$ is pseudo-contractive with modulus $c\in(0,1)$ if $\mbox{Fix}\T\ne \emptyset$ and for any $\bx^\star\in Fix\T$ and $\bx\in\R^d$,
    \begin{equation}
        \|\T(\bx)-\T(\bx^\star)\| \le c\|\bx-\bx^\star\|.
    \end{equation}
\end{definition}
Examples of averaged operators include the proximal operator $\prox_{f}$, of a closed and convex function $f$, the gradient descent operator $\Id-\gamma \nabla f$, $\gamma\in(0,2/L)$ of a convex and $L$-smooth $f$, the Douglas-Rachford splitting of two $1/2$-averaged operators and the forward-backward splitting of a maximally monotone operator and a cocoercive operator. These operators may be pseudo-contractive under stronger conditions \cite{bauschke2011convex}.

\begin{theorem}\label{thm:boundedgeneral} 
    Let $\bx^{\star}\in  Fix\T$ and $\{\bx(k)\}$ be generated by DEGAS under Assumptions~\ref{asm:blockdelayindependent}--\ref{asm:boundedelay}. If $T$ is averaged, then
    \begin{equation}\label{eq:1overkconverge}
    	\min_{t\le k} \mathbb{E}\left[\|(\Id-\T)(\mb{x}(t))\|\right] = O(1/k).
    \end{equation}
	If $T$ is also pseudo-contractive with modulus $c\in(0,1)$ then 
	\begin{equation}\label{eq:boundedgenerallinear}
	    \mathbb{E}[\|\mathbf{x}(k)-\mathbf{x}^\star\|^2] \le \rho_a^k \|\mathbf{x}(0)-\mathbf{x}^\star\|^2
	\end{equation}
	holds for all $k\in \N_0$ where $\rho_a=(1-\frac{1-c^2}{m})^{\frac{1}{1+\bar{\tau}/m}}$.
\end{theorem}

\begin{proof}
See Appendix \ref{appendix:proofthmboundegeneral}.
\end{proof}
In Theorem \ref{thm:boundedgeneral}, the expectation is taken over historical block selections. The linear rate \eqref{eq:boundedgenerallinear} is derived by using Lemma \ref{lemma:sequencelemma} in Appendix \ref{appendix:proofthmboundegeneral}, which establishes a linear convergence rate for a class of asynchronous sequences that \emph{significantly sharpens} the rate in \cite{Feyzmahdavian21}.


The rate in Theorem \ref{thm:boundedgeneral} is tight in the sense that it is of the same order as the best-known rates for the \emph{centralized} coordinate update \eqref{eq:coordinateupdate}. When $\bar{\tau}=0$, the rate $\rho_a$ in \eqref{eq:boundedgenerallinear} reduces to the typical rate $\rho_c:=1-\frac{1-c^2}{m}$ of the centralized coordinate update. Moreover, for any $\epsilon>0$, to achieve $\mathbb{E}[\|\bx(k)-\bx^\star\|^2]\le \epsilon$, DEGAS requires at most
\begin{equation}\label{eq:ratetocomplexity}
    K_a(\epsilon):=(1+\bar{\tau}/m)K_c(\epsilon)
\end{equation}
iterations, where $K_c(\epsilon)=\log_{1/\rho_c}\frac{\|\bx(0)-\bx^\star\|^2}{\epsilon}$ is the iteration complexity of the centralized coordinate update method for achieving the same accuracy.

\begin{rem}[Linear speed-up]
Suppose that $\bar{\tau}$ is proportional to the number $n$ of workers. This happens when workers are updated in a cyclic order, and is a good approximation for many distributed architectures for small to moderate values of $n$.
Then, 
 by \eqref{eq:ratetocomplexity},
\begin{equation}\label{eq:complexity}
    K_a(\epsilon)=(1+\Theta(n)/m)K_c(\epsilon).
\end{equation}
If the computation time of $\T_i$ dominates the per-iteration cost of DEGAS, then a single iteration of DEGAS takes $1/n$ of the time of a centralized coordinate update~\eqref{eq:coordinateupdate} \cite{Peng16}. Combining this observation with~\eqref{eq:complexity} reveals that DEGAS needs
\begin{equation}\label{eq:ratiotime}
    1/n+\Theta(1)/m
\end{equation}
times that of the centralized coordinate update to achieve a given accuracy. 
Note that \eqref{eq:ratiotime} is approximately inversely proportional to $n$ when $m\gg n$. This phenomenon  is called linear speedup \cite{Peng16} and is a desirable property of distributed optimization/learning algorithms.
\end{rem}

\begin{rem}[Comparison with ARock]\label{rem:ARockratecom}
\citet{Peng16} establish a linear convergence rate $O(\rho^k)$ for ARock that is improved in~\cite{Feyzmahdavian21} to 
\begin{equation}\label{eq:rhoinrem}
    \rho = 1-\frac{1-c^2}{m(1+6\left(\frac{\bar{\tau}}{m}+\sqrt{\frac{\bar{\tau}}{m}}\right))}\ge \rho_a^{\frac{1+\bar{\tau}/m}{1+6\left(\frac{\bar{\tau}}{m}+\sqrt{\frac{\bar{\tau}}{m}}\right)}}.
\end{equation}
The inequality in \eqref{eq:rhoinrem} is established in Appendix \ref{appendix:proofofrem}. Hence, to achieve the same accuracy, ARock needs at least $\frac{1+6(\bar{\tau}/m+\sqrt{\bar{\tau}/m})}{1+\bar{\tau}/m}$ times as many iterations of DEGAS. When $\bar{\tau}=m$, this is a factor of roughly $6.5$.
\end{rem}
\subsubsection{Self-adaptivity to actual delays}

Since the maximal delay can be very large while most delays are significantly smaller \cite{mishchenko2022asynchronous, Wu2022adaptive, koloskovasharper}, the ability to adapt to the actual delays and not be significantly slowed down by infrequent occurrences of larger delays is an attractive algorithm feature. We call this property  
\emph{self-adaptivity} of an asynchronous algorithm to actual delays.

Unlike ARock, whose maximally allowable step-size decreases with the maximum delay, DEGAS includes no delay information in its parameters and intuitively has better self-adaptivity. However, in Theorem \ref{thm:boundedgeneral}, the use of worst-case delay in the analysis can give loose convergence rate bounds and does not indicate any advantage of a system
in which the worst-case delay is rarely attained over
one that tends to run with delays close to the worst-case
all the time.
To reveal how the actual delays rather than their upper bound affects the convergence of DEGAS, we consider delays described by the following stochastic model: 
\begin{assumption}\label{asm:stochasticdelay}
    The delays $\{\tau(k)\}_{k\in\N_0}$ are i.i.d. with probability distribution $\mc{P}$, where
    \begin{equation}\label{eq:delaydistribution}
        \operatorname{Pr}(\tau(k)=i) = P_i,~\forall i\in \{0,\ldots,\bar{\tau}\},
    \end{equation}
    with $\sum_{i=0}^{\bar{\tau}} P_i = 1$ and $P_i\ge 0$ $\forall i\in\{0,\ldots,\bar{\tau}\}$.
\end{assumption}
As the next result shows, the convergence rate of Algorithm~\ref{alg:Coor} under such delays can be characterized by
\begin{equation*}
    \phi(\rho)=\rho-\rho_c-\frac{c^2}{m}(\sum_{i=0}^{\bar{\tau}}P_i\rho^{-i}-1).
\end{equation*}
%
\begin{theorem}\label{thm:adaptbounded}
    Suppose that Assumption \ref{asm:stochasticdelay} holds, $\bar{\tau}\ge 1$, and $P_0<1$. Let $\{\bx(k)\}$ be generated by DEGAS. If $\T$ is pseudo-contractive with modulus $c\in(0,1)$ then
	\begin{equation}\label{eq:adapt}
	    \mathbb{E}[\|\mathbf{x}(k)-\mathbf{x}^\star\|^2] \le \rho_{\mathcal P}^k\|\mathbf{x}(0)-\mathbf{x}^\star\|^2,~\forall k\in \N_0,
	\end{equation}
	where $\rho_{\mc{P}}=\alpha_{\mc{P}}\rho_a+(1-\alpha_{\mc{P}})\rho_c$ with \begin{equation}\label{eq:alphavalue}
	    \begin{split}
	       \alpha_{\mc{P}}&=\frac{1}{1-\phi(\rho_a)/\phi(\rho_c)}\in(0,1).
	    \end{split}
	\end{equation}
	In \eqref{eq:alphavalue}, $\phi(\rho_a)\ge 0$ and $\phi(\rho_c)<0$.
\end{theorem}

\begin{proof}
    See Appendix \ref{append:proofofthmbound}.
\end{proof}

\begin{figure*}[htb]
	\centering
	\caption{Adaptivity of DEGAS and ARock to real delays (small (bound), uniform (bound), and large (bound) are bounds in Theorem \ref{thm:adaptbounded}).}
    \vspace{-0.2cm}
	\subfigure[delay distribution]{
		\includegraphics[scale=0.7]{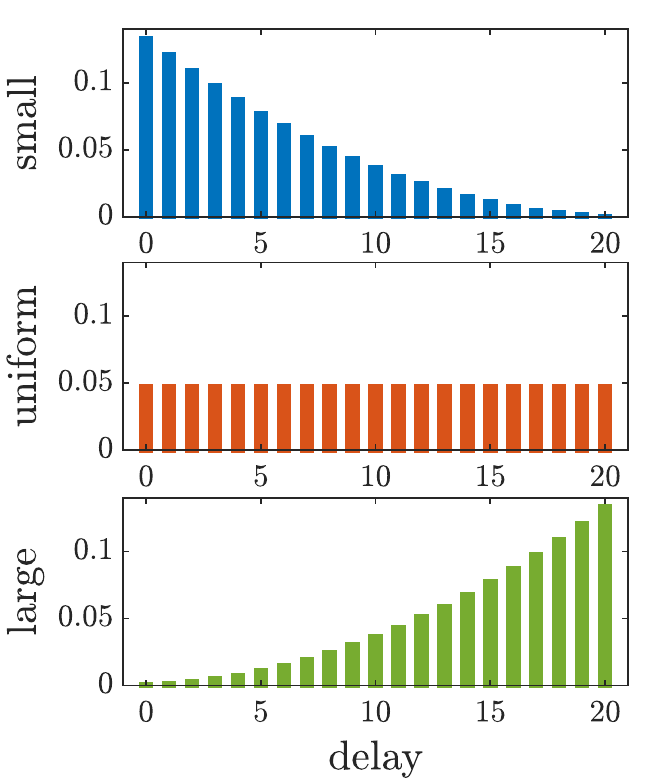}\label{fig:delay}}
	\subfigure[DEGAS]{
		\includegraphics[scale=0.63]{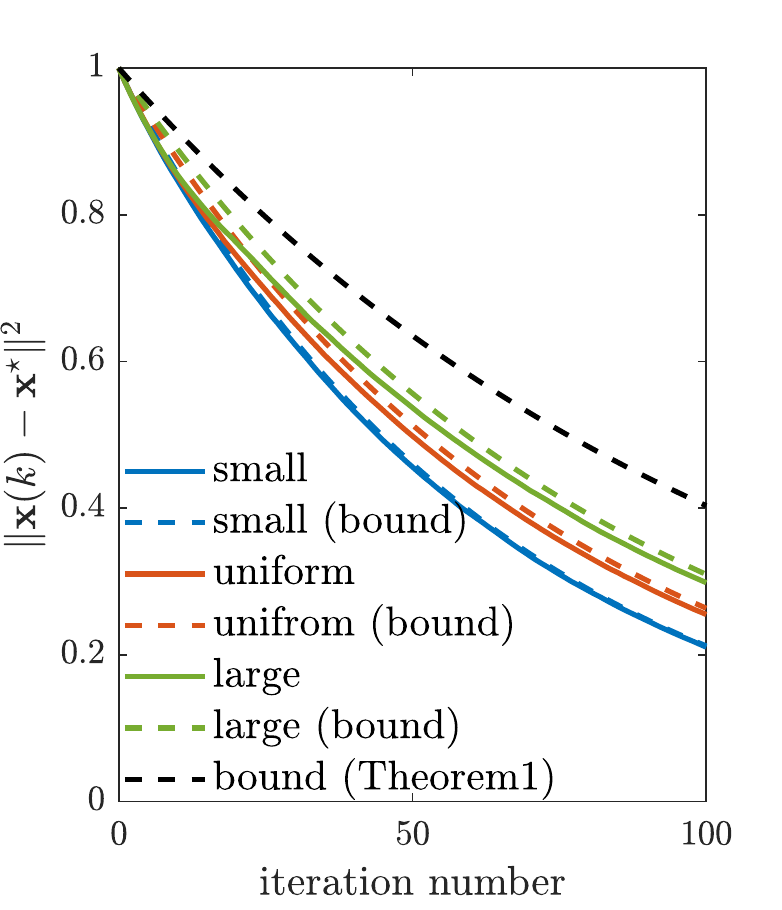}\label{fig:alg1}}
	\subfigure[ARock]{
		\includegraphics[scale=0.63]{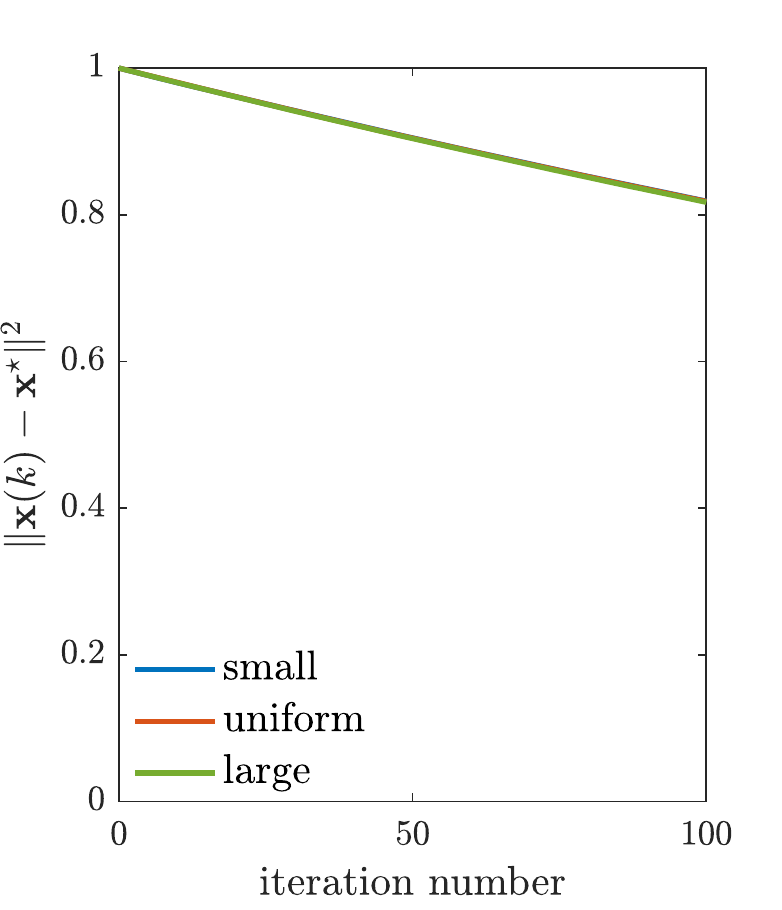}\label{fig:arock}}
	\label{fig:adaptivity}
\vspace{-0.1cm}
\end{figure*}

The expectation in Theorem \ref{thm:adaptbounded} is taken jointly over historical delays and block selection. A remarkable feature of Theorem~\ref{thm:adaptbounded} is that it allows to compute an explicit convergence rate bound for any given delay distribution ${\mathcal P}$. 
The convergence factor $\rho_{\mc{P}}$ is a convex combination of the corresponding quantities for the synchronous (centralized) and bounded-delay models, and the mixing parameter $\alpha_{\mc{P}}$ depends on the delay distribution $\mc{P}$. However, from \eqref{eq:alphavalue}, it is not straightforward to see how qualitative characteristics of the delay distribution (e.g., the mean or the variance) affect $\phi(\rho_c)$, $\phi(\rho_a)$, and $\alpha_{\mc{P}}$. 
As we will show next, such insight can be developed using the concept of stochastic dominance \cite{hadar1969rules}.
%

\textbf{Effect of delay under stochastic dominance}: Suppose that $\mathcal{P}$ and $\mathcal{P}^{'}$ are two probability distributions defined by \eqref{eq:delaydistribution}.
\begin{definition}[stochastic dominance]
We say $\mc{P}$ first-order stochastically dominates $\mc{P}'$ ($\mc{P}\succeq_1\mc{P}'$) if
\begin{equation*}
    \sum_{j=0}^i P_j\ge \sum_{j=0}^i P_j',\quad \forall i\in \{0,\ldots, \bar{\tau}\},
\end{equation*}
i.e., $\mc{P}$ always has a larger or equal cumulative probability.
\end{definition}
The stochastic dominance model compares the proportion of small delays in two delay distributions, which is different but has close connections to the mean-variance model: 
\begin{proposition}\label{rem:stochasticdominance}
Suppose that $\mc{P}\succeq_1\mc{P}'$. Then, the mean value of $\mc{P}$ is smaller than or equal to that of $\mc{P}'$ and if they share the same mean value, then the variance of $\mc{P}$ is smaller than or equal to that of $\mc{P}'$.
\end{proposition}
\begin{proof}
    These results are straightforward to derive from \cite{hadar1969rules}. For completeness, we provide a simple proof in in Appendix \ref{append:proofofremarkstochasticdominance}.
\end{proof}

Below we show the impact of the delay on the convergence of DEGAS using the stochastic dominance model.
\begin{lemma}\label{lemma:stochasticdominance}
If $\mc{P}\succeq_1\mc{P}'$, then $\rho_{\mc{P}}\le\rho_{\mc{P}'}$.
\end{lemma}
\begin{proof}
See Appendix \ref{append:proofoflemmastochasticdominance}.
\end{proof}
By Lemma \ref{lemma:stochasticdominance}, for a given delay bound, a larger proportion of small delays yields faster convergence of DEGAS.

\textbf{Demonstration with a simple operator:} To demonstrate the self-adaptivity of DEGAS to actual delays, we consider the simple operator $\T(\bx)=0.8\bx$ and three stochastic delay models. Note that when the delays are generated by stochastic models, the number of workers does not affect the update \eqref{eq:Alg1formal} or the convergence of DEGAS in terms of iteration index. We choose $m=20$, $\bar{\tau}=20$, and each block $x_i\in\mathbb{R}$. The stochastic models are: For any $i\in [0,\bar{\tau}]$, 1) \textbf{small}: $P_i=\frac{(\bar{\tau}+1-i)^2}{\sum_{j=1}^{\bar{\tau}+1}j^2}$;  \textbf{uniform}: $P_i=\frac{1}{\bar{\tau}+1}$;  \textbf{large}: $P_i=\frac{(i+1)^2}{\sum_{j=1}^{\bar{\tau}+1}j^2}$. If the generated value is larger than $k$, we set $\tau(k)=k$ to ensure $\tau(k)\in [0,k]$  We also run ARock with the same operator for comparison. For ARock \eqref{eq:ARock}, we fine tune $\gamma(k)$ in its theoretical range $(0,\frac{1}{2\bar{\tau}/\sqrt{m}+1})$ in \cite{Peng16} which is broader than that in \cite{Feyzmahdavian21} in the experiment setting, and draw blocks $i$ to update uniformly at random.

We execute $2000$ runs of both algorithms, each for $100$ iterations, and plot the average result. We also plot the rate bounds in Theorems \ref{thm:boundedgeneral}--\ref{thm:adaptbounded} to test their tightness. The result is shown in Figure \ref{fig:adaptivity}. We can see that DEGAS significantly converges slower when the proportion of small delays decreases (small $\rightarrow$ uniform $\rightarrow$ large), indicating the excellent delay adaptivity of DEGAS. Moreover, the rate bound in Theorem \ref{thm:adaptbounded} is very tight for the simulated scenarios and can reflect the effect of the delay on the convergence. The gap between the bounds in Theorems \ref{thm:boundedgeneral} is not as tight as that in Theorem \ref{thm:adaptbounded} but is still tight. In contrast to DEGAS, no clear difference on the convergence speed of ARock under the three delay patterns can be observed due to its small step-size caused by the large $\bar{\tau}$. In addition, for all the three delay models, DEGAS is much faster than ARock. In passing, we note that even {\bf small} delay model is much less extreme than the actual delays reported in 
\cite{mishchenko2022asynchronous}

\subsubsection{Convergence on unbounded delay}

We also study DEGAS on unbounded delays and consider the total asynchrony model \cite{bertsekas2003parallel}.
\begin{assumption}[Total asynchrony]\label{asm:parallelasynchrony}
	The delay sequence $\{\tau(k)\}$ satisfies $\lim_{k\rightarrow +\infty} k-\tau(k) = +\infty$.
\end{assumption}
Assumption \ref{asm:parallelasynchrony} is very general and guarantees that old information must eventually be purged from the system.

\begin{theorem}\label{thm:general}
    Suppose that Assumptions \ref{asm:blockdelayindependent},\ref{asm:parallelasynchrony} hold and $\bx^\star\in Fix\T$. Let $\{\bx(k)\}$ be generated by DEGAS. If $\T$ is averaged, then $\lim_{k \rightarrow +\infty}\inf_{t\le k} \mathbb{E}\left[\|(\Id-\T)(\mb{x}(t))\|\right] = 0.$
\end{theorem}
\begin{proof}
See Appendix \ref{appendix:proofthmgeneral}.
\end{proof}

The expectation in Theorem \ref{thm:general} is taken over historical block selections. Under Assumption \ref{asm:parallelasynchrony}, it's impractical to derive explicit convergence rates due to the lack of bounds on the growing speed of delays. 

Below, we derive explicit convergence rates for DEGAS under the following delay model which satisfies Assumption \ref{asm:parallelasynchrony} but has a sublinearly or linearly growing delay bound.

\begin{assumption}[sublinear \& linear delay]\label{asm:sublineardelay}
    There exist $\eta \in(0,1)$, $\beta\in (0,1]$, and $\gamma\ge 1$ such that $\tau(k)\le \eta k^\beta+\gamma$ $\forall k\in\N_0$.
\end{assumption}

\begin{theorem}[sublinear convergence]\label{thm:unbounded}
    Suppose that Assumption \ref{asm:sublineardelay} holds. Let $\{\bx(k)\}$ be generated by DEGAS. If $\T$ is pseudo-contractive with modulus $c\in(0,1)$, then
    \begin{equation*}
        \mathbb{E}[\|\mathbf{x}(k)-\mathbf{x}^\star\|^2] = \begin{cases}
        O\left(\left(1-\frac{1-c^2}{m }\right)^{k^{1-\beta}}\right), & \beta\in(0,1),\\
        O(1/k), & \beta=1.
        \end{cases}
    \end{equation*}
\end{theorem}
\begin{proof}
    See Appendix \ref{sec:convanasublineargrow}.
\end{proof}
Like Theorems \ref{thm:boundedgeneral}, \ref{thm:general}, the expectation in Theorem \ref{thm:unbounded} is taken over historical block selections. Theorems \ref{thm:boundedgeneral}, \ref{thm:unbounded} display how delays affect the order of the convergence rate of DEGAS. Such a relationship is summarized in a more clear way in Table \ref{tab:relation}. Overall speaking, faster growing speed of delay leads to slower convergence, which coincides with intuition.

\begin{table}[!htb]
    \centering
    \begin{tabular}{c|c|c|c}
        \hline
        delay bound & $\bar{\tau}$ & $O(k^\beta)$, $\beta\in (0,1)$ & $O(k)$\\
        \hline
        rate & linear & $O(\rho_c^{k^{1-\beta}})$ & $O(1/k)$\\
        \hline
    \end{tabular}
    \caption{asynchrony and convergence rate ($\rho_c = 1-\frac{1-c^2}{m}$).}    \label{tab:relation}
\end{table}

\citet{wu2022optimal} derive the same order of results as Theorem \ref{thm:unbounded} for the asynchronous BCD specialized from ARock and prove that they are optimal in terms of the convergence rate order. \citet{hannah2018unbounded} show that ARock converges under certain stochastic and deterministic unbounded delay models. However, when considering deterministic delay models, neither of them guarantees the convergence of ARock under the more general total asynchrony assumption. Moreover, they both require carefully designed delay-dependent step-sizes to guarantee convergence, while DEGAS can converge with delay-free parameters.

\section{Applications}
By concretizing the operator $\T$ in DEGAS, we obtain novel and efficient asynchronous variants of BCD and ADMM.

\subsection{Delay-agnostic asynchronous BCD}\label{ssec:DTBCD}

BCD \cite{richtarik2014iteration} solves the composite optimization problem
\begin{equation}\label{eq:composite}
    \underset{\bx\in\mathbb{R}^d}{\operatorname{minimize}}~f(\bx)+\sum_{i=1}^m r_i(x_i),
\end{equation}
where $f:\mathbb{R}^d\rightarrow \mathbb{R}$ is convex and $L$-smooth,
$x_i\in\mathbb{R}^{d_i}$ is the $i$th block of $\bx$, \emph{i.e.}, $\bx=(x_1, \ldots, x_m)$, and each $r_i:\mathbb{R}^{d_i}\rightarrow \mathbb{R}\cup\{+\infty\}$ is closed and convex. At each iteration $k$, BCD chooses one block $i\in [m]$ and updates
\begin{equation}\label{eq:cpgd}
    x_i(k+1) = \prox_{\gamma r_i}(x_i(k) - \gamma \nabla_i f(\bx(k))),
\end{equation}
while $x_j(k+1)=x_j(k)$ for all $j\neq i$. Here, $\gamma>0$ is a step-size and $\nabla_i f(\cdot)$ is the partial gradient of $f$ with respect to $x_i$. This is equivalent to the coordinate update \eqref{eq:coordinateupdate} with \begin{equation}\label{eq:operatorcpgd}
    \T_i(\bx)= \prox_{\gamma r_i}(x_i - \gamma \nabla_i f(\bx)),~\forall i\in[m].
\end{equation}
We refer to DEGAS with $\T_i$ defined in \eqref{eq:operatorcpgd} as delay-agnostic asynchronous BCD. Compared to the existing asynchronous BCD \cite{Sun17,cheung2021fully} that specialized from ARock, the delay-agnostic asynchronous BCD enjoys the same advantages of DEGAS over ARock, i.e., delay-free parameters and nice convergence properties. 

If the optimal solution set of \eqref{eq:composite} is non-empty, so is $Fix \T$ where $\T=(\T_1,\ldots,\T_m)$ with each $\T_i$ given by \eqref{eq:operatorcpgd}, and every $\bx^\star\in Fix\T$ is an optimal solution of problem \eqref{eq:composite} \cite{bauschke2011convex}. Under proper conditions, $\T$ is averaged and pseudo-contractive, 
which implies convergence of delay-agnostic asynchronous BCD by the results in \S~\ref{ssec:convana}.

\begin{lemma}
Suppose that $\gamma\in(0,2/L)$.  The operator $\T$ defined in \eqref{eq:operatorcpgd} is $\alpha$-averaged with $\alpha=\frac{1}{\min(1,1/(L\gamma))+1/2}\in(0,1)$. If, in addition,
$f$ is $\mu$-strongly convex for some $\mu\in (0,L]$, then $\T$ is pseudo-contractive with modulus
$c=\sqrt{1-2\gamma\mu+\gamma^2\mu L}$.
\end{lemma}

\begin{proof}
The claim follows by Theorem 25.8 in \cite{bauschke2011convex} and Proposition 5 in \cite{Peng16}.
\end{proof}
Under the uniform random block selection rule and bounded delays, \cite{Sun17,cheung2021fully} establish, for the asynchronous BCD specialized from ARock, convergence rates of the same order as Theorem \ref{thm:boundedgeneral}. Moreover, neither of them requires Assumption \ref{asm:blockdelayindependent}, and \citet{Sun17} also consider stochastic and deterministic unbounded delays and deterministic block selection rules. However, compared to our delay-agnostic asynchronous BCD, the asynchronous BCD \cite{Sun17,cheung2021fully} inherits the disadvantages of ARock over DEGAS, \emph{i.e.}, delay-dependent step-sizes and the resulting slow convergence. Moreover, none of \cite{Sun17,cheung2021fully} provides convergence results under either of Assumptions \ref{asm:parallelasynchrony}--\ref{asm:sublineardelay}.


\subsection{Delay-agnostic asynchronous ADMM}
Consider the consensus optimization problem:
\begin{equation}\label{eq:consensusprob}
\begin{split}
    \underset{z\in\mathbb{R}^d}{\operatorname{minimize}}~&~\sum_{i=1}^m F_i(z),\\
\end{split}
\end{equation}
where each $F_i$ is convex and closed. Problem \eqref{eq:consensusprob} formulates some popular problems such as empirical risk minimization in machine learning \cite{boyd2011distributed} and, by letting $\bz=(z_1,\ldots,z_m)$, $F(\bz)=\sum_{i=1}^m F_i(z_i)$, and $\mc{C}=\{\bz:z_1 = \ldots=z_m\}$, it can be rewritten as
\begin{equation}\label{eq:ADMMprob}
    \underset{\bz\in\mathbb{R}^{md}}{\operatorname{minimize}}~F(\bz)+\mc{I}_{\mc{C}}(\bz),
\end{equation}
where $\mc{I}_{\mc{C}}$ is the indicator function of $\mc{C}$. One popular way of solving \eqref{eq:ADMMprob} is to use the update \eqref{eq:fpupdate} with $\T$ being the Douglas-Rachford splitting of $\partial \mc{I}_{\mc{C}}$ and $\partial F$ \cite{bauschke2011convex}, i.e.,
\begin{equation}\label{eq:ADMMoperator}
    \T=\Id+\lambda(\prox_{\gamma F}\circ (2\prox_{\gamma \mc{I}_{\mc{C}}}-\Id)-\prox_{\gamma \mc{I}_{\mc{C}}}),
\end{equation}
where $\lambda\in(0,2)$. If the optimal solution set of \eqref{eq:ADMMprob} is non-empty, so is $Fix \T$ and, for any $\bx^\star\in Fix \T$, $\prox_{\gamma \mc{I}_{\mc{C}}}(\mathbf{x}^\star)$ is an optimal solution of problem \eqref{eq:ADMMprob} \cite{bauschke2011convex}. We refer to DEGAS with $\T$ in \eqref{eq:ADMMoperator} as delay-agnostic asynchronous ADMM because its synchronous counterpart with $\lambda=1$ is equivalent to ADMM (see Appendix \ref{sec:ADMMequivalence}).

The delay-agnostic asynchronous ADMM can be asynchronously implemented as Algorithm \ref{alg:Coor}, where $\T_i(\bx^w)$ in step 6 can be computed by
\begin{align}
    & z_i = \frac{1}{m}\sum_{i=1}^m x_i^w,\label{eq:TiADMMz}\allowdisplaybreaks\\
    & \T_i(\bx^w) = x_i^w +\lambda(\prox_{\gamma F_i} (2z_i-x_i^w)-z_i).\label{eq:TiADMMx}
\end{align}

Below we show $\T$ in \eqref{eq:ADMMoperator} is an averaged operator under proper conditions, so that by the results in \S~\ref{ssec:convana} the delay-agnostic asynchronous ADMM converges under both bounded and unbounded delays.
\begin{lemma}\label{lemma:ADMMlemma}
The operator $\T$ in \eqref{eq:ADMMoperator} is $\lambda/2$-averaged.
\end{lemma}
\begin{proof}
See Appendix \ref{ssec:proofofADMM}.
\end{proof}
For some special examples of $F_i$'s, e.g., each $F_i$ is the indicator function of a subspace and $\lambda=1$, the operator $\T$ in \eqref{eq:ADMMoperator} becomes pseudo-contractive \cite{Bauschke2014linear}. In such cases, the delay-agnostic asynchronous ADMM can achieve linear convergence for bounded delays by Theorems \ref{thm:boundedgeneral}--\ref{thm:adaptbounded}, and sublinear convergence for unbounded delays by Theorem \ref{thm:unbounded} in Appendix \ref{sec:convanasublineargrow}. Moreover, the delay-adaptivity can be seen straightforwardly from Theorem \ref{thm:adaptbounded}.

\begin{rem}
A closely related asynchronous ADMM is developed in \cite{zhang2014asynchronous}, which updates according to \eqref{eq:TiADMMz}--\eqref{eq:TiADMMx} with $\lambda=1$, but sets the number of workers to be identical to the number of blocks with each $z_i$ being updated by the worker $i$. To guarantee convergence, they assume at each iteration, the probability for each worker to return their local variable to the master is identical, which rarely holds in practice. Moreover, they only provide $O(1/k)$ convergence in terms of the running-average $\frac{1}{k}\sum_{t=0}^{k-1} \bz(t)$ when the delays are bounded and has no convergence guarantees on the last-iterate $\bz(k)$. However, we prove convergence of the last iterate for bounded (Theorem \ref{thm:boundedgeneral}) and unbounded delays (Theorem \ref{thm:general}). Such rates can be improved as discussed below Lemma \ref{lemma:ADMMlemma} when the operator is pseudo-contractive.
\end{rem}

\subsubsection{Extension to a more general problem}\label{sssec:extendedADMM}

We extend the delay-agnostic asynchronous ADMM to solve
\begin{equation}\label{eq:complicatedproblem}
    \underset{\bz\in \mc{Z}}{\operatorname{minimize}}~f(\bz)+r(\bz),
\end{equation}
where $\bz=(z_1,\ldots,z_m)$ with each $z_i\in\mathbb{R}^{d_i}$, $\mc{Z}\subseteq \mathbb{R}^{\sum_{i=1}^m d_i}$ is a convex and closed set and is easy to project, $f$ is convex and $L$-smooth, and $r(\bz)=\sum_{i=1}^m r_i(z_i)$ with each $r_i$ being a convex, closed, but possibly non-smooth function. We discuss some examples of $\mc{Z}$: i) When $\mc{Z}=\mathbb{R}^{\sum_{i=1}^m d_i}$, \eqref{eq:complicatedproblem} reduces to problem \eqref{eq:composite}; ii) When $\mc{Z}=\{\bz: z_1 = z_2 = \ldots = z_m\}$, the problem \eqref{eq:complicatedproblem} becomes consensus optimization, which is slightly general than problem \eqref{eq:ADMMprob} since the objective function is allowed to have a non-separable smooth component; iii) When $\mc{Z}=\{\bz: A\bz\le b\}$ or $\mc{Z}=\{\bz: A\bz=b\}$ for a matrix $A$ and a vector $b$, it becomes resource allocation \cite{lin2015global}.

To exploit the composite structure of the objective function, we replace $\prox_{\gamma F}$ in \eqref{eq:ADMMoperator} by
\begin{equation}\label{eq:tprime}
    \T' = \theta \prox_{\gamma r}\circ(\Id-\gamma\nabla f) + (1-\theta)\Id,
\end{equation}
where $\gamma\in(0,2/L)$ and $\theta = \frac{\min(1, 1/(L\gamma))+1/2}{2}$. We average the proximal gradient operator with $\Id$ to make $\T'$ a $1/2$-averaged operator, which will be further used in convergence analysis (see Lemma \ref{lemma:extendedADMM} later). We also replace $\mc{C}$ in \eqref{eq:ADMMoperator} by $\mc{Z}$. Then, the new operator takes this form: Let $\lambda\in(0,2)$,
\begin{equation}\label{eq:TgeneralADMM}
    \T= \Id+\lambda(\T'\circ (2\prox_{\gamma \mc{I}_{\mc{Z}}}-\Id)-\prox_{\gamma \mc{I}_{\mc{Z}}}),
\end{equation}
which can be simplified to
\begin{equation}\label{eq:extendedsimpleADMMoperator}
    \begin{split}
        \T =& \prox_{\gamma r}\circ(\Id-\gamma\nabla f)\circ (2\prox_{\gamma \mc{I}_{\mc{Z}}}-\Id)\\
        &+\frac{2}{3}(\Id-\prox_{\gamma \mc{I}_{\mc{Z}}})
    \end{split}
\end{equation}
when $\gamma=1/L$ and $\lambda=4/3$. If the optimal solution set of \eqref{eq:complicatedproblem} is non-empty, so is $Fix \T$ with $\T$ in \eqref{eq:TgeneralADMM} and for any $\bx^\star\in Fix \T$, $\prox_{\gamma \mc{I}_{\mc{Z}}}(\mathbf{x}^\star)$ is an optimal solution of problem \eqref{eq:complicatedproblem} \cite{bauschke2011convex}.

We refer to DEGAS with $\T$ in \eqref{eq:TgeneralADMM} as extended delay-agnostic asynchronous ADMM, whose asynchronous implementation is straightforward to see from Algorithm \ref{alg:Coor}, where $\T_i(\bx^w)$ in step 6 can be computed by
\begin{align}
    z_i &= [\prox_{\gamma\mc{I}_{\mc{Z}}}(\bx^w)]_i,\allowdisplaybreaks\\
    y_i &= 2z_i-x_i^w,\label{eq:xiupdategeneralADMM}\allowdisplaybreaks\\
    \T_i(\bx^w) &= \lambda\theta\prox_{\gamma r_i}(y_i-\gamma\nabla_i f(y_i))\nonumber\\
    &+(1-(1-\theta)\lambda)x_i^w-\lambda(1-2(1-\theta))z_i.\label{eq:mainstepgeneralADMM}
\end{align}

\begin{lemma}\label{lemma:extendedADMM}
The operator $\T'$ in \eqref{eq:tprime} is $1/2$-averaged and $\T$ in \eqref{eq:TgeneralADMM} is $\lambda/2$-averaged.
\end{lemma}
\begin{proof}
See Appendix \ref{sec:proofofextendedADMM}.
\end{proof}

With Lemma \ref{lemma:extendedADMM}, convergence of the extended asynchronous ADMM for bounded and unbounded delays can be recovered straightforwardly from Theorems \ref{thm:boundedgeneral}, \ref{thm:general}, respectively.

\section{Experiments}\label{sec:experiment}

	\begin{figure*}[!htb]
		\centering
		\caption{Convergence for Lasso}
		\subfigure[theoretical parameters]{
		\includegraphics[scale=0.38]{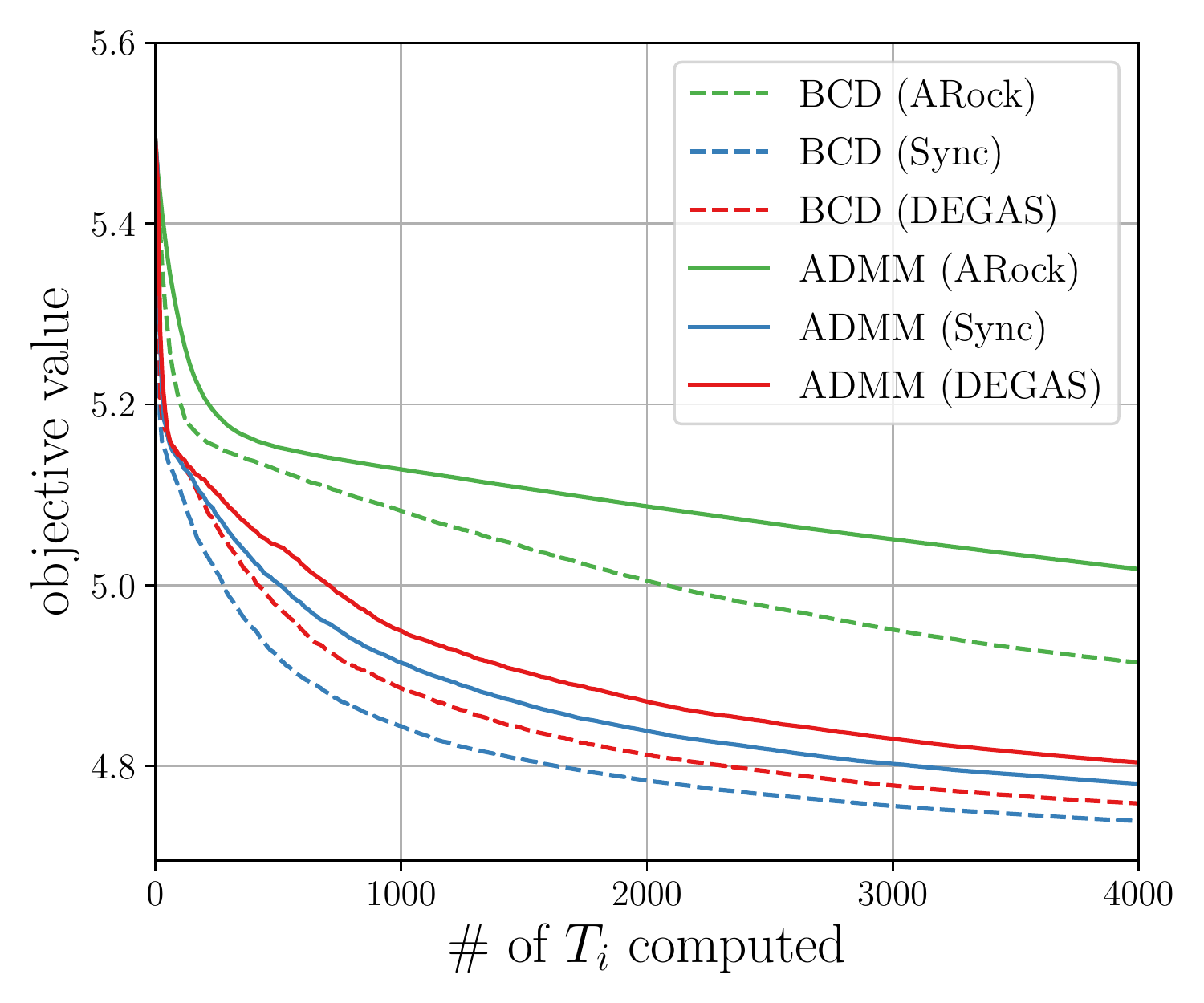}\label{fig:lassocifar_iter}}
        \subfigure[hand-tuned parameters]{
		\includegraphics[scale=0.38]{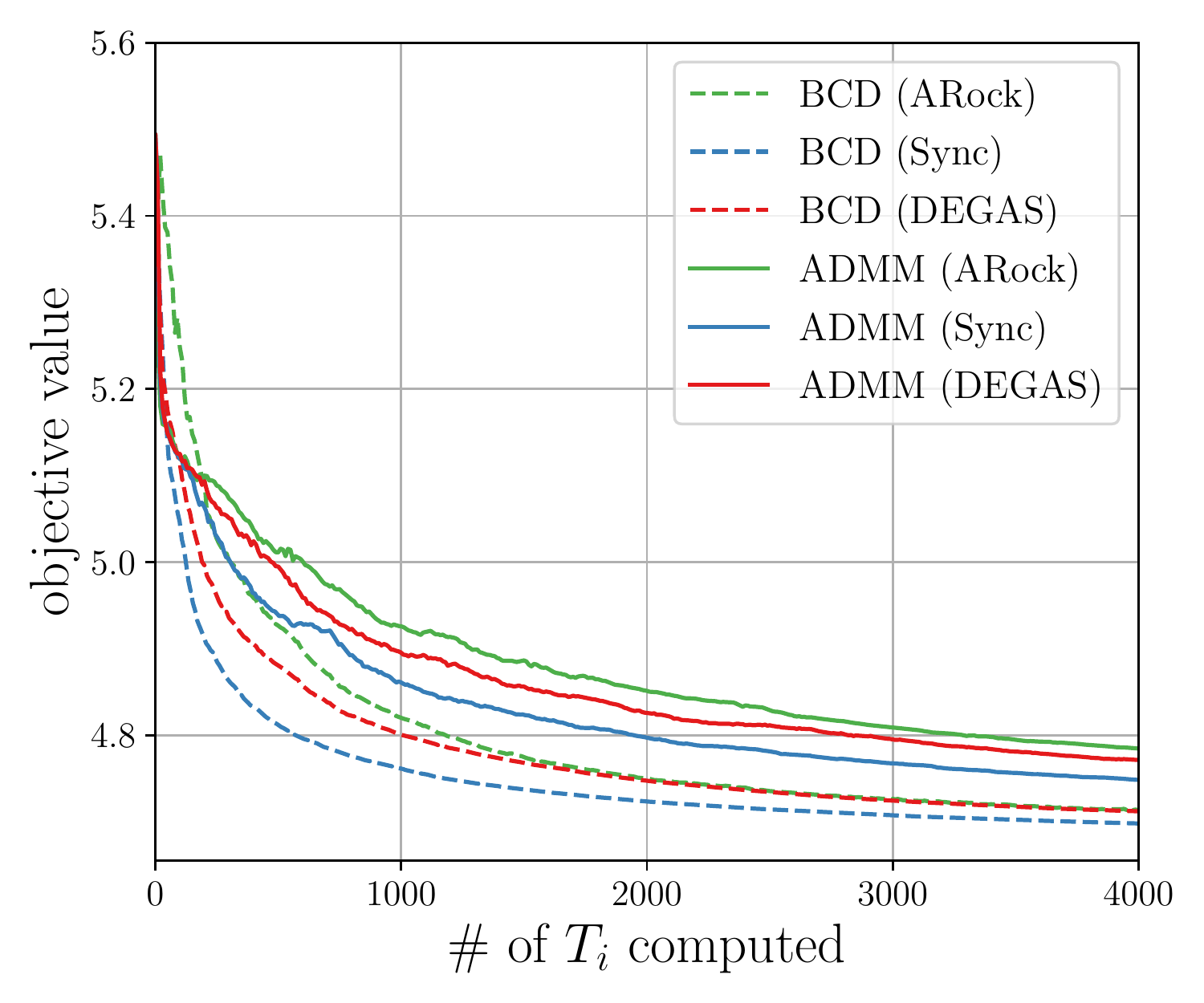}\label{fig:lassocifar_time}}
		\label{fig:LASSOCIFAR}
	\end{figure*}

    \begin{figure*}[htb]
		\centering
        \vspace{-0.6cm}
		\caption{Convergence for Logistic regression}
	\subfigure[theoretical parameters]{
	  \includegraphics[scale=0.38]{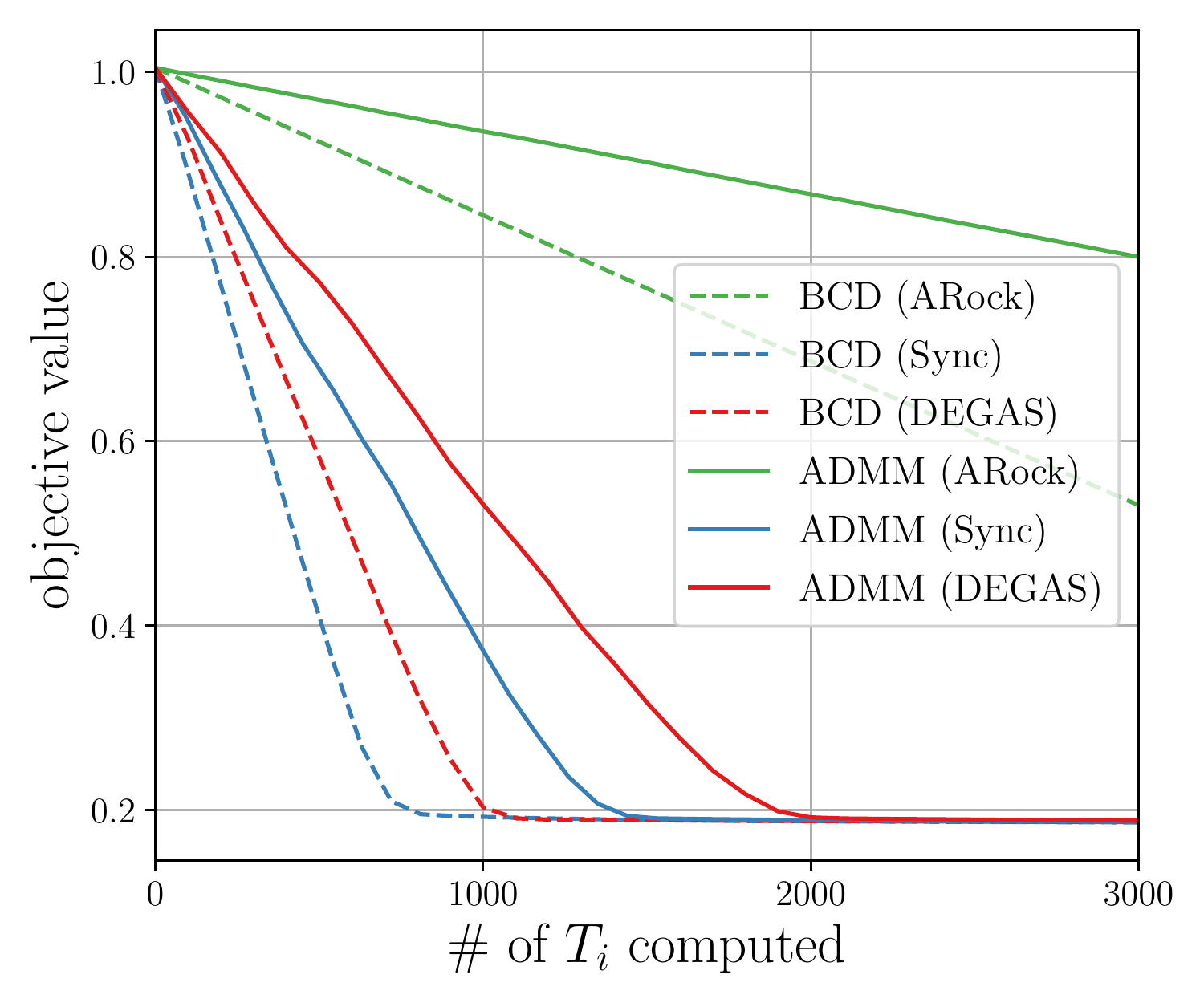}\label{fig:logisticcifar_iter}}
        \subfigure[hand-tuned parameters]{
	\includegraphics[scale=0.38]{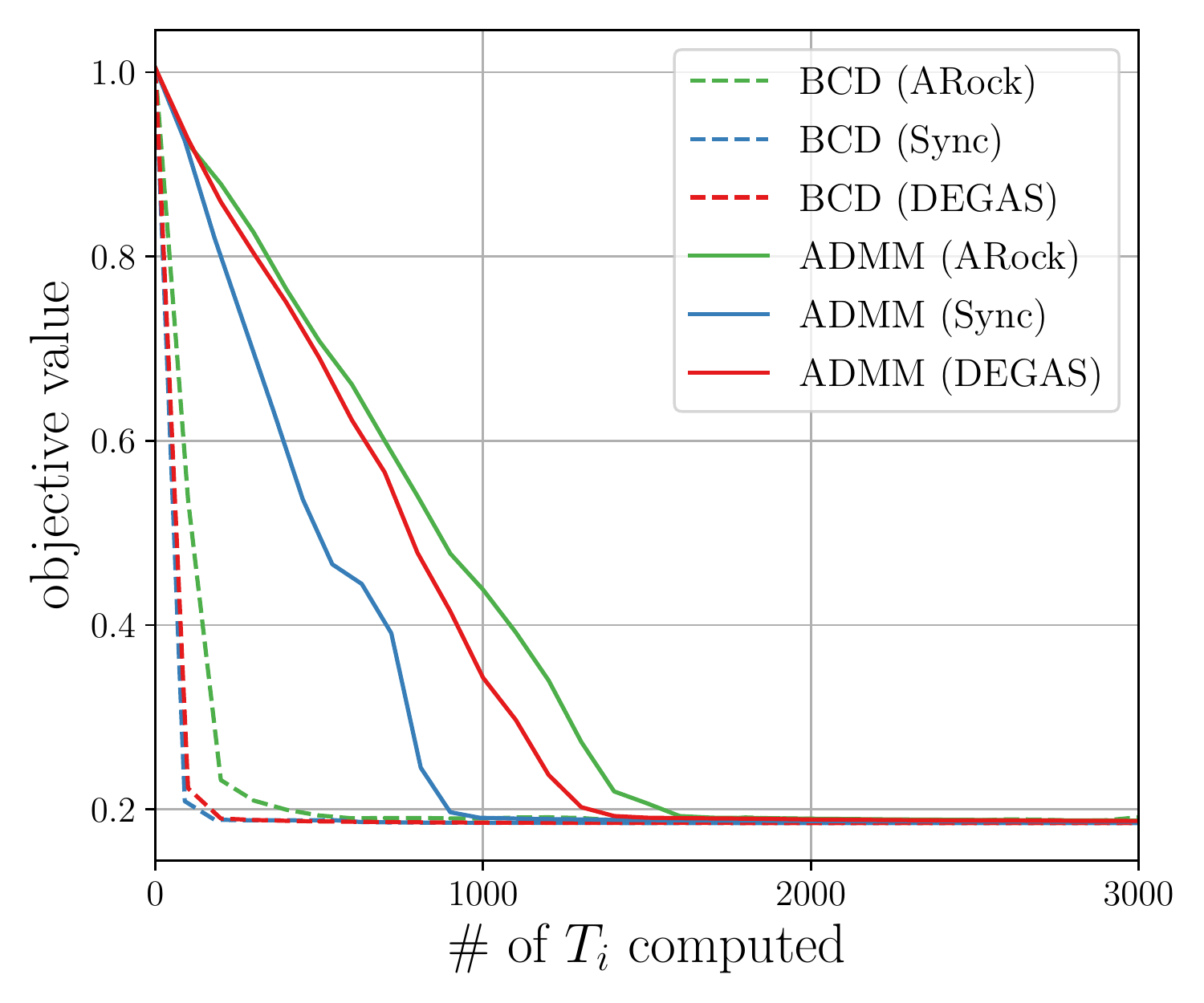}\label{fig:logisticcifar_time}}
		\label{fig:LogisticCIFAR}
  \vspace{-0.25cm}
\end{figure*}

We evaluate the practical performance of DEGAS on 
Lasso and regularized logistic regression problems on the CIFAR100 dataset~\cite{krizhevsky2009learning}.

Let $a_i$ be the feature of the $i$th sample, $b_i$ be the corresponding label, and $N$ be the number of samples. Then our test problems are on the form
\begin{equation}\label{eq:simuprob}
    \begin{split}
        \underset{x\in\mathbb{R}^d}{\operatorname{minimize}}~\frac{1}{N}\sum_{i=1}^N \ell_i(x)+\lambda_1\|x\|_1,
    \end{split}
\end{equation}
where $\ell_i(x) = \frac{1}{2}\|a_ix-b_i\|^2$ in Lasso and $\ell_i(x) = \log(1+e^{-b_i(a_i^Tx)})\!+\!\frac{\lambda_2}{2}\|x\|^2$ in regularized logistic regression.
We use $\lambda_1=10^{-3}$ 
and $\lambda_2=10^{-4}$.
We compare the proposed DEGAS with ARock and their common synchronous counterpart by solving \eqref{eq:simuprob}. In these methods, we choose the operator $\T$ as \eqref{eq:operatorcpgd} with $\gamma=1/L$ in BCD and \eqref{eq:extendedsimpleADMMoperator} in the extended ADMM. We set $m=20$ and implement all the methods on a 10-core machine ($1$ master and $9$ workers) using the message-passing framework MPI4py \cite{dalcin2008mpi}. {Note that we do not assume any delay model and all the delays are generated by real interactions between the master and workers.} We consider both theoretical and hand-tuned parameters. 
In the former setting, we fine-tune the step-size of ARock within its theoretical range in \cite{Feyzmahdavian21} which is broader than that in \cite{Peng16} in the experiment setting, and the other two methods include no parameters to tune. In the hand-tune step-size setting, we run all the methods for finding the fixed point of $\operatorname{Id}+\lambda(\T-\operatorname{Id})$, $\lambda>0$ and tune $\lambda$.


We plot the convergence in the number of computed $\T_i$ in Figures \ref{fig:LASSOCIFAR}--\ref{fig:LogisticCIFAR}, from which we make the following observation: 1) For both theoretical and hand-tuned step-sizes, DEGAS is much faster than ARock in all tested scenarios, which demonstrates its superior performance compared to ARock; 2) the synchronous method outperforms DEGAS in terms of the number of $\T_i$ computation. However, as asynchronous methods can complete more computations within a fixed time interval compared to synchronous methods, DEGAS may converge faster in wall-clock time, which is discussed in Appendix \ref{sec:exprunningtime}. We also observe that DEGAS and the synchronous methods can converge with much larger hand-tuned step-sizes than ARock. 
We plot the delay distribution generated by the experiments in Appendix \ref{sec:delaydistribution}.

	\section{Conclusion}
    We have proposed a delay-agnostic asynchronous coordinate update (DEGAS) method to find fixed-points of operators, which may have broad applications to algebra, optimization, and game theory. Compared to the alternative method ARock that can only converge under a delay-dependent parameter condition, DEGAS can converge under a delay-free parameter condition. Moreover, DEGAS can adapt well to the actual delays and converge significantly faster than ARock in both the settings of theoretical and hand-tuned parameters according to our numerical experiments. 
    	\section*{Acknowledgements}
	This work was supported by WASP and the Swedish Research Council (Vetenskapsr\r{a}det) under grants 2019-05319 and 2020-03607. We thank the anonymous reviewers for their detailed and valuable feedback.
	
\newpage
\newpage

\bibliography{reference.bib}

\begin{thebibliography}{50}
\providecommand{\natexlab}[1]{#1}
\providecommand{\url}[1]{\texttt{#1}}
\expandafter\ifx\csname urlstyle\endcsname\relax
  \providecommand{\doi}[1]{doi: #1}\else
  \providecommand{\doi}{doi: \begingroup \urlstyle{rm}\Url}\fi

\bibitem[Agarwal \& Duchi(2011)Agarwal and Duchi]{agarwal2011distributed}
Agarwal, A. and Duchi, J.~C.
\newblock Distributed delayed stochastic optimization.
\newblock \emph{Advances in neural information processing systems}, 24, 2011.

\bibitem[Aviv et~al.(2021)Aviv, Hakimi, Schuster, and Levy]{aviv21a}
Aviv, R.~Z., Hakimi, I., Schuster, A., and Levy, K.~Y.
\newblock Asynchronous distributed learning : Adapting to gradient delays
  without prior knowledge.
\newblock In \emph{Proceedings of the 38th International Conference on Machine
  Learning}, volume 139, pp.\  436--445, 2021.

\bibitem[Aybat et~al.(2015)Aybat, Wang, and Iyengar]{aybat2015asynchronous}
Aybat, N., Wang, Z., and Iyengar, G.
\newblock An asynchronous distributed proximal gradient method for composite
  convex optimization.
\newblock In \emph{International Conference on Machine Learning}, pp.\
  2454--2462. PMLR, 2015.

\bibitem[Aytekin et~al.(2016)Aytekin, Feyzmahdavian, and Johansson]{aytekin16}
Aytekin, A., Feyzmahdavian, H.~R., and Johansson, M.
\newblock Analysis and implementation of an asynchronous optimization algorithm
  for the parameter server.
\newblock \emph{arXiv preprint arXiv:1610.05507}, 2016.

\bibitem[Bauschke et~al.(2011)Bauschke, Combettes, et~al.]{bauschke2011convex}
Bauschke, H.~H., Combettes, P.~L., et~al.
\newblock \emph{Convex analysis and monotone operator theory in Hilbert
  spaces}, volume 408.
\newblock Springer, 2011.

\bibitem[Bauschke et~al.(2014)Bauschke, {Bello Cruz}, Nghia, Phan, and
  Wang]{Bauschke2014linear}
Bauschke, H.~H., {Bello Cruz}, J., Nghia, T.~T., Phan, H.~M., and Wang, X.
\newblock The rate of linear convergence of the douglas–rachford algorithm
  for subspaces is the cosine of the friedrichs angle.
\newblock \emph{Journal of Approximation Theory}, 185:\penalty0 63--79, 2014.

\bibitem[Bertsekas(1983)]{bertsekas1983distributed}
Bertsekas, D.~P.
\newblock Distributed asynchronous computation of fixed points.
\newblock \emph{Mathematical Programming}, 27\penalty0 (1):\penalty0 107--120,
  1983.

\bibitem[Bertsekas \& Tsitsiklis(2003)Bertsekas and
  Tsitsiklis]{bertsekas2003parallel}
Bertsekas, D.~P. and Tsitsiklis, J.~N.
\newblock Parallel and distributed computation: numerical methods.
\newblock 2003.

\bibitem[Boyd et~al.(2011)Boyd, Parikh, Chu, Peleato, Eckstein,
  et~al.]{boyd2011distributed}
Boyd, S., Parikh, N., Chu, E., Peleato, B., Eckstein, J., et~al.
\newblock Distributed optimization and statistical learning via the alternating
  direction method of multipliers.
\newblock \emph{Foundations and Trends{\textregistered} in Machine learning},
  3\penalty0 (1):\penalty0 1--122, 2011.

\bibitem[Cheung et~al.(2021)Cheung, Cole, and Tao]{cheung2021fully}
Cheung, Y.~K., Cole, R., and Tao, Y.
\newblock Fully asynchronous stochastic coordinate descent: a tight lower bound
  on the parallelism achieving linear speedup.
\newblock \emph{Mathematical Programming}, 190:\penalty0 615--677, 2021.

\bibitem[Cohen et~al.(2021)Cohen, Daniely, Drori, Koren, and
  Schain]{cohen2021asynchronous}
Cohen, A., Daniely, A., Drori, Y., Koren, T., and Schain, M.
\newblock Asynchronous stochastic optimization robust to arbitrary delays.
\newblock \emph{Advances in Neural Information Processing Systems},
  34:\penalty0 9024--9035, 2021.

\bibitem[Dalc{\'\i}n et~al.(2008)Dalc{\'\i}n, Paz, Storti, and
  D’El{\'\i}a]{dalcin2008mpi}
Dalc{\'\i}n, L., Paz, R., Storti, M., and D’El{\'\i}a, J.
\newblock Mpi for python: Performance improvements and mpi-2 extensions.
\newblock \emph{Journal of Parallel and Distributed Computing}, 68\penalty0
  (5):\penalty0 655--662, 2008.

\bibitem[Feyzmahdavian \& Johansson(2021)Feyzmahdavian and
  Johansson]{Feyzmahdavian21}
Feyzmahdavian, H.~R. and Johansson, M.
\newblock Asynchronous iterations in optimization: New sequence results and
  sharper algorithmic guarantees.
\newblock \emph{arXiv preprint arXiv:2109.04522}, 2021.

\bibitem[Feyzmahdavian et~al.(2014)Feyzmahdavian, Aytekin, and
  Johansson]{feyzmahdavian2014delayed}
Feyzmahdavian, H.~R., Aytekin, A., and Johansson, M.
\newblock A delayed proximal gradient method with linear convergence rate.
\newblock In \emph{2014 IEEE International Workshop on Machine Learning for
  Signal Processing (MLSP)}, pp.\  1--6. IEEE, 2014.

\bibitem[Frommer(1991)]{FROMMER1991105}
Frommer, A.
\newblock Generalized nonlinear diagonal dominance and applications to
  asynchronous iterative methods.
\newblock \emph{Journal of Computational and Applied Mathematics}, 38\penalty0
  (1):\penalty0 105--124, 1991.

\bibitem[Glasgow \& Wootters(2022)Glasgow and
  Wootters]{glasgow2022asynchronous}
Glasgow, M.~R. and Wootters, M.
\newblock Asynchronous distributed optimization with stochastic delays.
\newblock In \emph{International Conference on Artificial Intelligence and
  Statistics}, pp.\  9247--9279, 2022.

\bibitem[Hadar \& Russell(1969)Hadar and Russell]{hadar1969rules}
Hadar, J. and Russell, W.~R.
\newblock Rules for ordering uncertain prospects.
\newblock \emph{The American economic review}, 59\penalty0 (1):\penalty0
  25--34, 1969.

\bibitem[Hale et~al.(2017)Hale, Nedi{\'c}, and Egerstedt]{hale2017asynchronous}
Hale, M.~T., Nedi{\'c}, A., and Egerstedt, M.
\newblock Asynchronous multiagent primal-dual optimization.
\newblock \emph{IEEE Transactions on Automatic Control}, 62\penalty0
  (9):\penalty0 4421--4435, 2017.

\bibitem[Hannah \& Yin(2018)Hannah and Yin]{hannah2018unbounded}
Hannah, R. and Yin, W.
\newblock On unbounded delays in asynchronous parallel fixed-point algorithms.
\newblock \emph{Journal of Scientific Computing}, 76\penalty0 (1):\penalty0
  299--326, 2018.

\bibitem[Iutzeler et~al.(2020)Iutzeler, Malick, and
  de~Oliveira]{iutzeler2020asynchronous}
Iutzeler, F., Malick, J., and de~Oliveira, W.
\newblock Asynchronous level bundle methods.
\newblock \emph{Mathematical Programming}, 184\penalty0 (1):\penalty0 319--348,
  2020.

\bibitem[Koloskova et~al.(2022)Koloskova, Stich, and Jaggi]{koloskovasharper}
Koloskova, A., Stich, S.~U., and Jaggi, M.
\newblock Sharper convergence guarantees for asynchronous sgd for distributed
  and federated learning.
\newblock In \emph{Advances in Neural Information Processing Systems}, 2022.

\bibitem[Krizhevsky et~al.(2009)Krizhevsky, Hinton,
  et~al.]{krizhevsky2009learning}
Krizhevsky, A., Hinton, G., et~al.
\newblock Learning multiple layers of features from tiny images.
\newblock 2009.

\bibitem[Leblond et~al.(2018)Leblond, Pedregosa, and
  Lacoste-Julien]{leblond:18}
Leblond, R., Pedregosa, F., and Lacoste-Julien, S.
\newblock Improved asynchronous parallel optimization analysis for stochastic
  incremental methods.
\newblock \emph{{Journal of Machine Learning Research}}, 2018.

\bibitem[Li et~al.(2013)Li, Zhou, Yang, Li, Xia, Andersen, and Smola]{LiM13}
Li, M., Zhou, L., Yang, Z., Li, A., Xia, F., Andersen, D.~G., and Smola, A.
\newblock Parameter server for distributed machine learning.
\newblock In \emph{Big Learning NIPS Workshop}, volume~6, pp.\ ~2, 2013.

\bibitem[Lian et~al.(2018)Lian, Zhang, Zhang, and Liu]{lian2018asynchronous}
Lian, X., Zhang, W., Zhang, C., and Liu, J.
\newblock Asynchronous decentralized parallel stochastic gradient descent.
\newblock In \emph{International Conference on Machine Learning}, pp.\
  3043--3052. PMLR, 2018.

\bibitem[Lin et~al.(2015)Lin, Ma, and Zhang]{lin2015global}
Lin, T., Ma, S., and Zhang, S.
\newblock On the global linear convergence of the {ADMM} with multiblock
  variables.
\newblock \emph{SIAM Journal on Optimization}, 25\penalty0 (3):\penalty0
  1478--1497, 2015.

\bibitem[Liu \& Wright(2015)Liu and Wright]{liu2015}
Liu, J. and Wright, S.~J.
\newblock Asynchronous stochastic coordinate descent: Parallelism and
  convergence properties.
\newblock \emph{SIAM Journal on Optimization}, 25\penalty0 (1):\penalty0
  351--376, 2015.

\bibitem[Liu et~al.(2014)Liu, Wright, R{\'e}, Bittorf, and Sridhar]{liu2014}
Liu, J., Wright, S., R{\'e}, C., Bittorf, V., and Sridhar, S.
\newblock An asynchronous parallel stochastic coordinate descent algorithm.
\newblock In \emph{International Conference on Machine Learning}, pp.\
  469--477. PMLR, 2014.

\bibitem[Luo et~al.(2020)Luo, He, Zhuo, and Qian]{luo2020prague}
Luo, Q., He, J., Zhuo, Y., and Qian, X.
\newblock Prague: High-performance heterogeneity-aware asynchronous
  decentralized training.
\newblock In \emph{Proceedings of the Twenty-Fifth International Conference on
  Architectural Support for Programming Languages and Operating Systems}, pp.\
  401--416, 2020.

\bibitem[Mania et~al.(2017)Mania, Pan, Papailiopoulos, Recht, Ramchandran, and
  Jordan]{Mania17}
Mania, H., Pan, X., Papailiopoulos, D., Recht, B., Ramchandran, K., and Jordan,
  M.~I.
\newblock Perturbed iterate analysis for asynchronous stochastic optimization.
\newblock \emph{SIAM Journal on Optimization}, 27\penalty0 (4):\penalty0
  2202--2229, 2017.

\bibitem[Mehyar et~al.(2007)Mehyar, Spanos, Pongsajapan, Low, and
  Murray]{mehyar2007asynchronous}
Mehyar, M., Spanos, D., Pongsajapan, J., Low, S.~H., and Murray, R.~M.
\newblock Asynchronous distributed averaging on communication networks.
\newblock \emph{IEEE/ACM Transactions On Networking}, 15\penalty0 (3):\penalty0
  512--520, 2007.

\bibitem[Mishchenko et~al.(2018)Mishchenko, Iutzeler, Malick, and
  Amini]{mishchenko2018delay}
Mishchenko, K., Iutzeler, F., Malick, J., and Amini, M.-R.
\newblock A delay-tolerant proximal-gradient algorithm for distributed
  learning.
\newblock In \emph{International Conference on Machine Learning}, pp.\
  3587--3595. PMLR, 2018.

\bibitem[Mishchenko et~al.(2022)Mishchenko, Bach, Even, and
  Woodworth]{mishchenko2022asynchronous}
Mishchenko, K., Bach, F., Even, M., and Woodworth, B.
\newblock Asynchronous {SGD} beats minibatch {SGD} under arbitrary delays.
\newblock In \emph{Advances in Neural Information Processing Systems}, 2022.

\bibitem[Moallemi \& Van~Roy(2010)Moallemi and
  Van~Roy]{moallemi2010convergence}
Moallemi, C.~C. and Van~Roy, B.
\newblock Convergence of min-sum message-passing for convex optimization.
\newblock \emph{IEEE Transactions on Information Theory}, 56\penalty0
  (4):\penalty0 2041--2050, 2010.

\bibitem[Nesterov(2012)]{Nesterov12}
Nesterov, Y.
\newblock Efficiency of coordinate descent methods on huge-scale optimization
  problems.
\newblock \emph{SIAM Journal on Optimization}, 22\penalty0 (2):\penalty0
  341--362, 2012.

\bibitem[Peng et~al.(2016)Peng, Xu, Yan, and Yin]{Peng16}
Peng, Z., Xu, Y., Yan, M., and Yin, W.
\newblock {AR}ock: an algorithmic framework for asynchronous parallel
  coordinate updates.
\newblock \emph{SIAM Journal on Scientific Computing}, 38\penalty0
  (5):\penalty0 A2851--A2879, 2016.

\bibitem[Recht et~al.(2011)Recht, Re, Wright, and Niu]{recht2011hogwild}
Recht, B., Re, C., Wright, S., and Niu, F.
\newblock Hogwild!: A lock-free approach to parallelizing stochastic gradient
  descent.
\newblock \emph{Advances in Neural Information Processing Systems},
  24:\penalty0 693--701, 2011.

\bibitem[Richt{\'a}rik \& Tak{\'a}{\v{c}}(2014)Richt{\'a}rik and
  Tak{\'a}{\v{c}}]{richtarik2014iteration}
Richt{\'a}rik, P. and Tak{\'a}{\v{c}}, M.
\newblock Iteration complexity of randomized block-coordinate descent methods
  for minimizing a composite function.
\newblock \emph{Mathematical Programming}, 144\penalty0 (1):\penalty0 1--38,
  2014.

\bibitem[Richt{\'a}rik \& Tak{\'a}{\v{c}}(2016)Richt{\'a}rik and
  Tak{\'a}{\v{c}}]{richtarik2016distributed}
Richt{\'a}rik, P. and Tak{\'a}{\v{c}}, M.
\newblock Distributed coordinate descent method for learning with big data.
\newblock \emph{The Journal of Machine Learning Research}, 17\penalty0
  (1):\penalty0 2657--2681, 2016.

\bibitem[Soori et~al.(2020)Soori, Mishchenko, Mokhtari, Dehnavi, and
  Gurbuzbalaban]{soori20a}
Soori, S., Mishchenko, K., Mokhtari, A., Dehnavi, M.~M., and Gurbuzbalaban, M.
\newblock {DAve-QN}: A distributed averaged quasi-newton method with local
  superlinear convergence rate.
\newblock In \emph{International Conference on Artificial Intelligence and
  Statistics}, pp.\  1965--1976. PMLR, 2020.

\bibitem[Sra et~al.(2016)Sra, Yu, Li, and Smola]{sra2016adadelay}
Sra, S., Yu, A.~W., Li, M., and Smola, A.
\newblock Adadelay: Delay adaptive distributed stochastic optimization.
\newblock In \emph{Artificial Intelligence and Statistics}, pp.\  957--965.
  PMLR, 2016.

\bibitem[Sun et~al.(2017)Sun, Hannah, and Yin]{Sun17}
Sun, T., Hannah, R., and Yin, W.
\newblock Asynchronous coordinate descent under more realistic assumption.
\newblock In \emph{Proceedings of the 31st International Conference on Neural
  Information Processing Systems}, pp.\  6183--6191, 2017.

\bibitem[Sun et~al.(2019)Sun, Sun, Li, and Liao]{Sun19}
Sun, T., Sun, Y., Li, D., and Liao, Q.
\newblock General proximal incremental aggregated gradient algorithms: Better
  and novel results under general scheme.
\newblock \emph{Advances in Neural Information Processing Systems},
  32:\penalty0 996--1006, 2019.

\bibitem[Wright(2015)]{wright2015coordinate}
Wright, S.~J.
\newblock Coordinate descent algorithms.
\newblock \emph{Mathematical Programming}, 151\penalty0 (1):\penalty0 3--34,
  2015.

\bibitem[Wu et~al.(2017)Wu, Yuan, Ling, Yin, and Sayed]{wu2017decentralized}
Wu, T., Yuan, K., Ling, Q., Yin, W., and Sayed, A.~H.
\newblock Decentralized consensus optimization with asynchrony and delays.
\newblock \emph{IEEE Transactions on Signal and Information Processing over
  Networks}, 4\penalty0 (2):\penalty0 293--307, 2017.

\bibitem[Wu et~al.(2022{\natexlab{a}})Wu, Magn\'{u}sson, Feyzmahdavian, and
  Johansson]{Wu2022adaptive}
Wu, X., Magn\'{u}sson, S., Feyzmahdavian, H.~R., and Johansson, M.
\newblock Delay-adaptive step-sizes for asynchronous learning.
\newblock In \emph{International Conference on Machine Learning}, pp.\
  24093--24113. PMLR, 2022{\natexlab{a}}.

\bibitem[Wu et~al.(2022{\natexlab{b}})Wu, Magn{\'u}sson, Feyzmahdavian, and
  Johansson]{wu2022optimal}
Wu, X., Magn{\'u}sson, S., Feyzmahdavian, H.~R., and Johansson, M.
\newblock Optimal convergence rates of totally asynchronous optimization.
\newblock In \emph{2022 IEEE 61st Conference on Decision and Control (CDC)},
  pp.\  6484--6490. IEEE, 2022{\natexlab{b}}.

\bibitem[Zhang \& You(2019)Zhang and You]{zhang2019asyspa}
Zhang, J. and You, K.
\newblock Asyspa: An exact asynchronous algorithm for convex optimization over
  digraphs.
\newblock \emph{IEEE Transactions on Automatic Control}, 65\penalty0
  (6):\penalty0 2494--2509, 2019.

\bibitem[Zhang \& Kwok(2014)Zhang and Kwok]{zhang2014asynchronous}
Zhang, R. and Kwok, J.
\newblock Asynchronous distributed {ADMM} for consensus optimization.
\newblock In \emph{International Conference on Machine Learning}, pp.\
  1701--1709. PMLR, 2014.

\bibitem[Zhou et~al.(2018)Zhou, Mertikopoulos, Bambos, Glynn, Ye, Li, and
  Fei-Fei]{zhou2018distributed}
Zhou, Z., Mertikopoulos, P., Bambos, N., Glynn, P., Ye, Y., Li, L.-J., and
  Fei-Fei, L.
\newblock Distributed asynchronous optimization with unbounded delays: How slow
  can you go?
\newblock In \emph{International Conference on Machine Learning}, pp.\
  5970--5979. PMLR, 2018.

\end{thebibliography}
\bibliographystyle{icml2023}

\onecolumn

\appendix

\section{Proof of Theorem \ref{thm:boundedgeneral}}\label{appendix:proofthmboundegeneral}

At each iteration $k\in\N_0$, each $i\in [m]$ has equal probability to be selected as $i(k)$. Then by \eqref{eq:Alg1formal}, \begin{equation}\label{eq:nonincreasingdis}
	\begin{split}
		& \mathbb{E}[\|\bx(k+1)-\bx^\star\|^2]\\
		= & \sum_{i=1}^m \frac{1}{m}\left(\|\T_i(\bx(k-\tau(k)))-x_i^\star\|^2+\sum_{j\ne i}\|x_j(k)-x_j^\star\|^2\right)\\
		= & \frac{1}{m}\|\T(\bx(k-\tau(k)))-\bx^\star\|^2+\left(1-\frac{1}{m}\right)\|\bx(k)-\bx^\star\|^2,
	\end{split}
\end{equation}
where $x_i^\star\in\R^{d_i}$ is the $i$th block of $\bx^\star$ and the expectation is taken over the block $i(k)$. Taking \eqref{eq:nonincreasingdis} at hand, we are ready to prove the results for both averaged and pseudo-contractive operator $\T$.

\subsection{Proof for averaged $\T$}

For all $k\in\N_0$, let
\begin{align*}
    V(k)&=\mathbb{E}[\|\bx(k)-\bx^\star\|^2],\\ W(k)&=\frac{1-\alpha}{m\alpha}\mathbb{E}[\|(\operatorname{Id}-\T)(\bx(k))\|^2],
\end{align*}
where the expectations are taken over the historical updated blocks $\{i(t)\}_{t\le k}$.
For all $t\in\N_0$, define 
\begin{align*}
    \mc{A}(t)&=[t(\bar{\tau}+1), (t+1)(\bar{\tau}+1)),\\
    a(t)&=\operatorname{\arg\;\max}_{k\in \mc{A}(t)} V(k).
\end{align*}
The proof includes three steps. Step 1 establishes the relationship between $\{V(k)\}$ and $\{W(k)\}$:
\begin{equation}\label{eq:boundedLyapunovfirm}
	V(k+1) \le\max_{\max(0, k-\bar{\tau})\le \ell\le k} V(\ell) - W(k-\tau(k)),\quad\forall k\in\N_0.
\end{equation}
Based on \eqref{eq:boundedLyapunovfirm}, step 2 proves
\begin{equation}\label{eq:Vmonotone}
    V(a(t)) \le V(a(t-1))- W(a(t)-1-\tau(a(t)-1)),\quad \forall t\in\N,
\end{equation}
which is then used to derive \eqref{eq:1overkconverge} in step 3.

\textbf{Step 1}: The proof uses Proposition 4.25 in \cite{bauschke2011convex}: Denote the average parameter of $\T$ as $\alpha\in (0,1)$. Then, for any $\bx,\by\in\R^d$,
\begin{equation}\label{eq:prop425}
\begin{split}
    &\|\T(\bx)-\T(\by)\|^2\le \|\bx-\by\|^2-\frac{1-\alpha}{\alpha}\|(\Id-\T)(\bx)-(\Id-\T)(\by)\|^2.
\end{split}
\end{equation}
Substituting $\bx=\bx(k-\tau(k))$ and $\by=\bx^\star$ into \eqref{eq:prop425} and using $\T(\bx^\star)=\bx^\star$, we have
\begin{equation}\label{eq:applyprop425}
\begin{split}
    &\|\T(\bx(k-\tau(k)))-\bx^\star\|^2\\
    \le& \|\bx(k-\tau(k))-\bx^\star\|^2-\frac{1-\alpha}{\alpha}\|(\Id-\T)(\bx(k-\tau(k)))\|^2.
\end{split}
\end{equation}
Substituting \eqref{eq:applyprop425} into \eqref{eq:nonincreasingdis} and taking expectation on both sides of the resulting equation yields
\begin{equation}\label{eq:oldboundedLyapunovfirm}
	\begin{split}
		V(k+1) \le& \frac{V(k-\tau(k))}{m}+\left(1-\frac{1}{m}\right)V(k)-W(k-\tau(k))\\
		\le& \max_{\max(0, k-\tau(k))\le \ell\le k} V(\ell) - W(k-\tau(k)).
	\end{split}
\end{equation}
By \eqref{eq:oldboundedLyapunovfirm} and $\tau(k)\le\bar{\tau}$ assumed in Assumption \ref{asm:boundedelay}, we have \eqref{eq:boundedLyapunovfirm}.

\textbf{Step 2}: We first show by induction that for any $k\in\N_0$ satisfying $k+1\in \mc{A}(t)$ or equivalently, $k\in [t(\bar{\tau}+1)-1, (t+1)(\bar{\tau}+1)-1)$,
\begin{equation}\label{eq:maximalinthepast}
    \max_{\max(0, k-\bar{\tau})\le \ell\le k} V(\ell) \le \max_{\ell\in \mc{A}(t-1)} V(\ell).
\end{equation}
When $k=t(\bar{\tau}+1)-1$, since $[\max(0,k-\bar{\tau}),k]=\mc{A}(t-1)$, the equation \eqref{eq:maximalinthepast} holds. Suppose that \eqref{eq:maximalinthepast} holds at $k=k'-1$ for some $k'$ satisfying $k'+1\in \mc{A}(t)$. Then, by letting $k=k'-1$ in both \eqref{eq:boundedLyapunovfirm} and \eqref{eq:maximalinthepast}, we have \begin{equation*}
    \begin{split}
        V(k')&\overset{\eqref{eq:boundedLyapunovfirm}}{\le} \max_{\max(0, k'-1-\bar{\tau})\le \ell\le k'-1} V(\ell) - W(k'-1-\tau(k'-1))\\
        &\le \max_{\max(0, k'-1-\bar{\tau})\le \ell\le k'-1} V(\ell)\\
        &\overset{\eqref{eq:maximalinthepast}}{\le}  \max_{\ell\in \mc{A}(t-1)} V(\ell).
    \end{split}
\end{equation*}
This, together with \eqref{eq:maximalinthepast} at $k=k'-1$, yields \eqref{eq:maximalinthepast} at $k=k'$. Following this induction procedure we obtain \eqref{eq:maximalinthepast} for all $k$ satisfying $k+1\in \mc{A}(t)$. Then, by letting $k=a(t)-1$ in both \eqref{eq:boundedLyapunovfirm} and \eqref{eq:maximalinthepast}, we have
\begin{equation*}
\begin{split}
    V(a(t)) \overset{\eqref{eq:boundedLyapunovfirm}}{\le}& \max_{\max(0, a(t)-1-\bar{\tau})\le \ell\le a(t)-1} V(\ell) - W(a(t)-1-\tau(a(t)-1))\\
    \overset{\eqref{eq:maximalinthepast}}{\le}& \max_{\ell\in \mc{A}(t-1)} V(\ell) - W(a(t)-1-\tau(a(t)-1))\\
    =& V(a(t-1))- W(a(t)-1-\tau(a(t)-1)),
\end{split}
\end{equation*}
\emph{i.e.}, \eqref{eq:Vmonotone} holds.

\textbf{Step 3}: For any $t'\in\N$, by adding \eqref{eq:Vmonotone} from $t=1$ to $t=t'$, we have
\begin{equation}\label{eq:sumboundedw}
\begin{split}
    \sum_{t=1}^{t'} W(a(t)-1-\tau(a(t)-1))&\le V(a(0))-V(a(t'))\\
    &\le V(a(0))\\
    &=V(0),
\end{split}
\end{equation}
where $V(a(0))=V(0)$ can be easily derived from \eqref{eq:boundedLyapunovfirm}. For any $k\in\N_0$, let $t'(k)=\lfloor k/(\bar{\tau}+1)\rfloor-1$. Then, for all $t\in [t'(k)]$,
\begin{equation}\label{eq:atandk}
\begin{split}
    a(t)-1-\tau(a(t)-1) &\le a(t)\\
&\le a(t'(k))\\
&\le (t'(k)+1)(\bar{\tau}+1)\\
&\le k,
\end{split}
\end{equation}
which, together with \eqref{eq:sumboundedw}, yields
\begin{equation*}
\begin{split}
    \min_{1\le \ell \le k} W(\ell) &\overset{\eqref{eq:atandk}}{\le} \min_{1\le t\le t'(k)} W(a(t)-1-\tau(a(t)-1))\\
    &\le \frac{1}{t'(k)}\sum_{t=1}^{t'(k)} W(a(t)-1-\tau(a(t)-1))\\
    &\overset{\eqref{eq:sumboundedw}}{\le} \frac{V(0)}{t'(k)}.
\end{split}
\end{equation*}
Moreover, $t'(k) \ge k/(\bar{\tau}+1)-2$. Then, we have
\begin{equation*}
\begin{split}
    \min_{1\le \ell \le k} W(\ell) &\le \frac{V(0)}{k/(\bar{\tau}+1)-2}\\
    &= O(1/k),
\end{split}
\end{equation*}
\emph{i.e.}, \eqref{eq:1overkconverge} holds.

\subsection{Proof for pseudo-contractive $\T$}
By the pseudo-contractivity of $\T$,
\begin{equation*}
\begin{split}
    \|\T(\bx(k-\tau(k)))-\bx^\star\|^2\le c^2\|\bx(k-\tau(k))-\bx^\star\|^2,
\end{split}
\end{equation*}
substituting which into \eqref{eq:nonincreasingdis} and taking expectation on both sides of the resulting equation ensures
\begin{equation}\label{eq:boundedcontractionLyap}
	V(k+1)\le \frac{c^2}{m}V(k-\tau(k))+\left(1-\frac{1}{m}\right)V(k).
\end{equation}
To proceed, we establish a sequence result in the following lemma.
\begin{lemma}\label{lemma:sequencelemma}
Suppose that the following holds for a non-negative sequence $\{V(k)\}$ and two positive constants $p,q\in(0,1)$ satisfying $p+q<1$:
\begin{equation}\label{eq:Vkseq}
    V(k+1)\le pV(k)+q \max_{\max(0,k-\bar{\tau})\le \ell\le k} V(\ell).
\end{equation}
Then,
\begin{equation}\label{eq:sequenceconv}
    V(k)\le \rho^kV(0),~\forall k\in\N_0,
\end{equation}
where
\begin{equation}\label{eq:smallrho}
\rho = (p+q)^{\frac{1}{1+(1-p)\bar{\tau}}}.
\end{equation}
\end{lemma}
\begin{proof}
We prove \eqref{eq:sequenceconv} by induction. Clearly, \eqref{eq:sequenceconv} holds for $k=0$. Suppose that \eqref{eq:sequenceconv} holds for all $k\in [0, k']$ for some $k'\in\N_0$. Then, by \eqref{eq:Vkseq},
\begin{equation*}
\begin{split}
    V(k'+1)\le (p\rho^{k'}+q\rho^{k'-\bar{\tau}})V(0).
\end{split}
\end{equation*}
Hence, to show \eqref{eq:sequenceconv} at $k=k'+1$, it suffices to prove $p\rho^{k'}+q\rho^{k'-\bar{\tau}}\le \rho^{k'+1}$, which is equivalent to \begin{equation}\label{eq:keyequationlinear}
    p+q\rho^{-\bar{\tau}}\le \rho.
\end{equation}
Therefore, if \eqref{eq:keyequationlinear} holds, so does \eqref{eq:sequenceconv} at $k=k'+1$. Following this induction procedure, we will have that \eqref{eq:sequenceconv} holds for all $k\in\N_0$.

Next, we prove \eqref{eq:keyequationlinear}, which includes three steps. Step 1 shows the equivalence between \eqref{eq:keyequationlinear} and
\begin{equation}\label{eq:importanteqrho}
    \left(\frac{p+q}{\rho}\right)^{\frac{1}{1-p}}\ge \frac{q}{\rho-p}.
\end{equation}
Step 2 proves the following inequality: For any $\alpha\in [p+q, 1]$,
\begin{equation}\label{eq:importanteq}
    \left(\frac{p+q}{\alpha}\right)^{\frac{1}{1-p}}\ge \frac{q}{\alpha-p}.
\end{equation}
Step 3 combines the first two steps and derives \eqref{eq:keyequationlinear}.

\textbf{Step 1}: The equation \eqref{eq:keyequationlinear} is equivalent to
\begin{equation}\label{eq:rhotaulowerbound}
    \rho^{\bar{\tau}}\ge \frac{q}{\rho-p}.
\end{equation}
By \eqref{eq:smallrho},
\begin{equation*}
\begin{split}
    \rho^{\bar{\tau}} &= (p+q)^{\frac{\bar{\tau}}{1+(1-p)\bar{\tau}}}\\
    &= (p+q)^{\frac{(1-p)\bar{\tau}}{1+(1-p)\bar{\tau}}\cdot\frac{1}{1-p}}\\
    &= (p+q)^{\left(1-\frac{1}{1+(1-p)\bar{\tau}}\right)\cdot\frac{1}{1-p}}\\
    &= \left((p+q)^{\left(1-\frac{1}{1+(1-p)\bar{\tau}}\right)}\right)^{\frac{1}{1-p}}\\
    &= \left(\frac{p+q}{(p+q)^{\frac{1}{1+(1-p)\bar{\tau}}}}\right)^{\frac{1}{1-p}}\\
    &\overset{\eqref{eq:smallrho}}{=}\left(\frac{p+q}{\rho}\right)^{\frac{1}{1-p}}.
\end{split}
\end{equation*}
Therefore, \eqref{eq:rhotaulowerbound} is equivalent to \eqref{eq:importanteqrho}, which, together with the equivalence between \eqref{eq:keyequationlinear} and \eqref{eq:rhotaulowerbound}, yields the equivalence between \eqref{eq:keyequationlinear} and \eqref{eq:importanteqrho}.

\textbf{Step 2}: By Bernoulli's inequality, for any $a\in(0,1],~b\ge 1$,
\begin{equation}\label{eq:bernouli}
\begin{split}
    a^b&=(1-(1-a))^b\\
    &\ge 1-(1-a)b.
\end{split}
\end{equation}
Letting $a=\frac{p+q}{\alpha}$ and $b=\frac{1}{1-p}$, we have
\begin{equation*}
    \left(\frac{p+q}{\alpha}\right)^{\frac{1}{1-p}}\ge 1-\frac{1-(p+q)/\alpha}{1-p}.
\end{equation*}
Therefore, \eqref{eq:importanteq} holds if
\begin{equation}\label{eq:1minusandqoveralphap}
    1-\frac{1-(p+q)/\alpha}{1-p} \ge \frac{q}{\alpha-p},
\end{equation}
which is equivalent to $h(\alpha):=p\alpha^2-(1+p+q)p\alpha+p(p+q)\le 0$. Note that
\begin{equation*}
    h(p+q) = 0,~~ h(1) = 0.
\end{equation*}
Let $\gamma=\frac{1-\alpha}{1-p-q}$ which satisfies $\gamma\in[0,1]$ due to $\alpha\in [p+q, 1]$. By the convexity of $h(\alpha)$, we have
\begin{equation*}
\begin{split}
    h(\alpha) &= h(\gamma(p+q)+(1-\gamma))\\
    &\le \gamma h(p+q)+(1-\gamma)h(1)\\
    &= 0,
\end{split}
\end{equation*}
which further indicates \eqref{eq:1minusandqoveralphap} and \eqref{eq:importanteq}.

\textbf{Step 3}: By \eqref{eq:importanteq} with $\alpha=\rho$, we obtain \eqref{eq:importanteqrho}. Then, by the equivalence between \eqref{eq:keyequationlinear} and \eqref{eq:importanteq}, the equation \eqref{eq:keyequationlinear} holds, which concludes the proof.

\end{proof}

Since $\tau(k)\le \bar{\tau}$, \eqref{eq:boundedcontractionLyap} yields \eqref{eq:Vkseq} with $p=1-\frac{1}{m}$ and $q=\frac{c^2}{m}$. Then by Lemma \ref{lemma:sequencelemma}, we obtain \eqref{eq:boundedgenerallinear}. This completes the proof.

\textbf{Remark on Lemma \ref{lemma:sequencelemma}}: The linear rate \eqref{eq:boundedgenerallinear} in Lemma \ref{lemma:sequencelemma} significantly improves the existing rate in \cite{Feyzmahdavian21} in the sense of tightness. Specifically, \cite{Feyzmahdavian21} proves that for any non-negative sequence $\{V(k)\}$ satisfying \eqref{eq:Vkseq}, \eqref{eq:sequenceconv} holds with
\begin{equation}\label{eq:hamidrho}
    \rho = (p+q)^{\frac{1}{\bar{\tau}+1}}.
\end{equation}
To distinguish $\rho$ in \eqref{eq:smallrho} and \eqref{eq:hamidrho}, we denote the former as $\rho_1$ and the latter as $\rho_2$. Since $p+q\in (0,1)$ and $\frac{1}{1+(1-p)\bar{\tau}}\ge \frac{1}{1+\bar{\tau}}$, it always holds that $\rho_1\le \rho_2$. The difference between $\rho_1$ and $\rho_2$ becomes clear when we look at their resulting iteration complexities, where a smaller iteration complexity indicates a tighter convergence rate bound. To guarantee $V(k)\le \epsilon$ for some $\epsilon>0$, \eqref{eq:sequenceconv} with $\rho=\rho_1$ and $\rho=\rho_2$ requires
\begin{align*}
K_1(\epsilon) &= \frac{(1+(1-p)\bar{\tau})\ln \frac{V(0)}{\epsilon}}{\ln \frac{1}{p+q}},\\
K_2(\epsilon) &= \frac{(1+\bar{\tau})\ln \frac{V(0)}{\epsilon}}{\ln \frac{1}{p+q}},
\end{align*}
respectively. Here, we can see that 
\begin{equation*}
    \frac{K_1(\epsilon)}{K_2(\epsilon)} = \frac{1+(1-p)\bar{\tau}}{1+\bar{\tau}},
\end{equation*}
which can be very small when $p$ is close to $1$ and $\bar{\tau}$ is large. For example, for equation \eqref{eq:boundedcontractionLyap} which yields \eqref{eq:Vkseq} with $p=1-\frac{1}{m}$, if we let $\bar{\tau}=m=20$, then
\begin{equation*}
    \frac{K_1(\epsilon)}{K_2(\epsilon)} = \frac{1+\bar{\tau}/m}{1+\bar{\tau}} = \frac{2}{21},
\end{equation*}
indicating that the rate \eqref{eq:sequenceconv} yield by $\rho=\rho_1$ is much tighter than that yield by $\rho=\rho_2$.

We also visualize the tightness of the two bounds by considering the convergence of a concrete sequence $\{V(k)\}$: $V(0)=1$. At each $k\in\N_0$, \eqref{eq:boundedcontractionLyap} holds with equality, where $m=\bar{\tau}=20$, $c=0.8$, and $\tau(k) = \min(k,\bar{\tau})$. For this example, \eqref{eq:Vkseq} holds with $p=1-\frac{1}{m}=0.95$ and $q=c^2/m=0.032$.

The convergence of $\{V(k)\}$ and its theoretical two bounds are displayed in Figure \ref{fig:tightness}. By \ref{fig:tightness} we can see that the rate bound in Lemma \ref{lemma:sequencelemma} is close to the practical convergence of $\{V(k)\}$ and is much tighter than the rate bound in \cite{Feyzmahdavian21}.

\begin{figure*}\label{fig:tightness}
    \centering
    \includegraphics[scale=0.9]{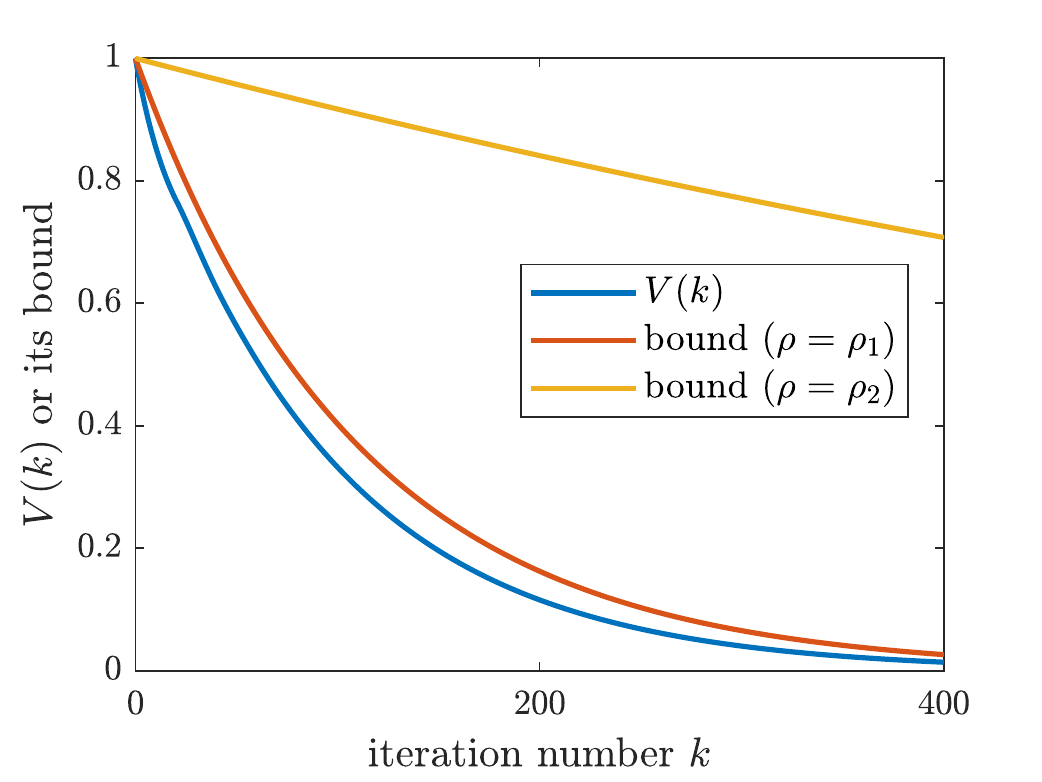}
    \caption{Tightness of rate bounds in Lemma \ref{lemma:sequencelemma} ($\rho=\rho_1$) and \cite{Feyzmahdavian21} ($\rho=\rho_2$).}
\end{figure*}
\section{Proof of \eqref{eq:rhoinrem} in Remark \ref{rem:ARockratecom}}\label{appendix:proofofrem}

The proof mainly uses Bernoulli's inequality: For any $a\in(0,1)$ and $b\in(0,1]$,
\begin{equation}\label{eq:abeq}
    a^b \le 1-b(1-a).
\end{equation}
Letting $a=1-\frac{1-c^2}{m} = \rho_a^{1+\bar{\tau}/m}$ and $b=\frac{1}{1+6\left(\frac{\bar{\tau}}{m}+\sqrt{\frac{\bar{\tau}}{m}}\right)}$ in \eqref{eq:abeq}, we have \eqref{eq:rhoinrem}.

\section{Proof of Theorem \ref{thm:adaptbounded}}\label{append:proofofthmbound}

We use the following lemma to prove the result.
\begin{lemma}\label{lemma:randomsequencelemma}
Suppose that the following holds for a non-negative sequence $\{V(k)\}$ and some non-negative constants $\sigma_i\in [0,1)$, $0\le i\le \bar{\tau}$ satisfying $\sigma_0<\sum_{i=0}^{\bar{\tau}} \sigma_i<1$:
\begin{equation}\label{eq:randomVkseq}
    V(k+1)\le \sum_{i=0}^{\bar{\tau}} \sigma_iV(k-i).
\end{equation}
Then,
\begin{equation}\label{eq:randomsequenceconv}
    V(k)\le \rho^kV(0),~\forall k\in\N_0,
\end{equation}
where
\begin{equation}\label{eq:delayadaptiverho}
\rho = \alpha' b+(1-\alpha')a.
\end{equation}
Here, $a=\sum_{i=0}^{\bar{\tau}} \sigma_i$, $b$ can be any scalar in $[b',1]$ where $b'=a^{\frac{1}{1+(1-\sigma_0)\bar{\tau}}}$, and
\begin{equation*}
    \alpha'=\frac{1}{1+\frac{b- \sum_{i=0}^{\bar{\tau}} \sigma_ib^{-i}}{\sum_{i=0}^{\bar{\tau}} \sigma_ia^{-i}-a}}\in[0,1].
\end{equation*}
In addition, $a-\sum_{i=0}^{\bar{\tau}} \sigma_ia^{-i}<0$ and $b- \sum_{i=0}^{\bar{\tau}} \sigma_ib^{-i}\ge 0$.
\end{lemma}
\begin{proof}
Define $h(s) = s-\sum_{i=0}^{\bar{\tau}} \sigma_is^{-i}$. Then,
\begin{equation}\label{eq:alphaprime}
    \alpha' = \frac{1}{1-h(b)/h(a)}.
\end{equation}
We first show that {$h(a)<0$} and $h(b)\ge 0$. Since $\sigma_0<\sum_{i=0}^{\bar{\tau}} \sigma_i$ and $a\in(0,1)$, we have $\sum_{i=1}^{\bar{\tau}} \sigma_i>0$ and $\sum_{i=1}^{\bar{\tau}} \sigma_i(1-a^{-i})<0$. Therefore,
\begin{align*}
    h(a) &= a-\sum_{i=0}^{\bar{\tau}} \sigma_ia^{-i}\\
    &= \sum_{i=0}^{\bar{\tau}} \sigma_i-\sum_{i=0}^{\bar{\tau}} \sigma_ia^{-i}\\
    &=\sum_{i=0}^{\bar{\tau}} \sigma_i(1-a^{-i})\\
    &=\sum_{i=1}^{\bar{\tau}} \sigma_i(1-a^{-i})\\
    &<0.
\end{align*}
Moreover, by $b'\in (0,1)$,
\begin{equation}\label{eq:hblb}
\begin{split}
    h(b') &= b'-\sum_{i=0}^{\bar{\tau}} \sigma_ib'^{-i}\\
    &= b'-\sigma_0-\sum_{i=1}^{\bar{\tau}} \sigma_i b'^{-i}\\
    &\ge b'-\sigma_0-\sum_{i=1}^{\bar{\tau}} \sigma_ib'^{-\bar{\tau}}\\
    &=b'-\sigma_0-(a-\sigma_0)b'^{-\bar{\tau}}.
\end{split}
\end{equation}
In addition, note that \eqref{eq:rhotaulowerbound} holds for $\rho$ in \eqref{eq:smallrho}. Then by letting $p=\sigma_0$ and $q=a-\sigma_0$ in \eqref{eq:rhotaulowerbound}, we have
\begin{equation*}
\begin{split}
    b'^{-\bar{\tau}} &\le\frac{b'-\sigma_0}{a-\sigma_0},
\end{split}
\end{equation*}
substituting which into \eqref{eq:hblb} gives $h(b')\ge 0$. In addition, $b\ge b'$ and $h(s)$ is an increasing function when $s>0$. Then we have $h(b)\ge h(b')\ge 0$.

Next, we prove \eqref{eq:randomsequenceconv} by induction.
Clearly, \eqref{eq:randomsequenceconv} holds for $k=0$. Suppose that \eqref{eq:randomsequenceconv} holds for all $k\in[0,k'-1]$ for some $k'\in\N$. Then, by \eqref{eq:randomVkseq},
\begin{equation*}
    V(k')\le \sum_{i=0}^{\bar{\tau}} \sigma_i\rho^{k'-i-1}V(0).
\end{equation*}
To prove \eqref{eq:randomsequenceconv} with $k=k'$, it suffices to show
\begin{equation*}
    \sum_{i=0}^{\bar{\tau}} \sigma_i\rho^{k'-i-1}\le \rho^{k'},
\end{equation*}
or equivalently,
\begin{equation}\label{eq:randrhoeq}
    \sum_{i=0}^{\bar{\tau}} \sigma_i\rho^{-i}\le \rho,
\end{equation}
which is equivalent to $h(\rho)\ge 0$. By $h(a)<0$, $h(b)\ge 0$, and \eqref{eq:alphaprime}, we have $\alpha'\in (0,1]$. Then, since $h$ is concave on the non-negative domain, we have
\begin{equation*}
\begin{split}
    h(\rho)&=h(\alpha' b+(1-\alpha')a)\\
    &\ge \alpha' h(b)+(1-\alpha')h(a)\\
    &=0,
\end{split}
\end{equation*}
which further yields \eqref{eq:randrhoeq} and also \eqref{eq:randomsequenceconv} with $k=k'$.
Following the induction procedure, we obtain \eqref{eq:randomsequenceconv} for all $k\in\N_0$.
\end{proof}

Taking expectation on both sides of \eqref{eq:boundedcontractionLyap} with respect to the delay $\tau(k)$ and using Assumption \ref{asm:stochasticdelay}, we have
\begin{equation*}
\begin{split}
    \mathbb{E}[V(k+1)]&\le \left(\frac{c^2P_0}{m}+1-\frac{1}{m}\right)V(k)+\sum_{i=1}^{\bar{\tau}}\frac{c^2P_i}{m} V(k-i).
\end{split}
\end{equation*}
Now we use Lemma \ref{lemma:randomsequencelemma} to derive the convergence of $\mathbb{E}[V(k)]$. Specifically, in Lemma \ref{lemma:randomsequencelemma} we let $\sigma_0=1-\frac{1}{m}+\frac{c^2P_0}{m}$, $\sigma_i=\frac{c^2P_i}{m}$ for all $i\ne 0$, and $b=\rho_a$, so that $a=\rho_c$ and $b'=\rho_c^{\frac{1}{1+\frac{(1-c^2P_0)\bar{\tau}}{m}}}\le \rho_c^{\frac{1}{1+\bar{\tau}/m}} = \rho_a$. Then by Lemma \ref{lemma:randomsequencelemma}, we obtain \eqref{eq:adapt}. Also by Lemma \ref{lemma:randomsequencelemma},
\begin{equation}\label{eq:phirhopositivenegative}
    \phi(\rho_c)<0,~~\phi(\rho_a)\ge 0.
\end{equation}


\section{Proof of Proposition \ref{rem:stochasticdominance}}\label{append:proofofremarkstochasticdominance}
For simplicity, we use $\text{MEAN}(\cdot)$ and $\text{VAR}(\cdot)$ to denote the mean value and variance of probability distributions, respectively. 

Letting $n=\bar{\tau}+1$, $x_i=i-1$, $\phi(x)=x$, $\alpha_i=P_{i-1}$, and $\beta_i=P_{i-1}'$ in Theorem 1 in \cite{hadar1969rules}, we have
\begin{equation*}
\begin{split}
    \text{MEAN}(\mc{P})&=\sum_{j=0}^{\bar{\tau}} jP_j\\
    &\le \sum_{j=0}^{\bar{\tau}} jP_j'\\
    &=\text{MEAN}(\mc{P}'),
\end{split}
\end{equation*}
\emph{i.e.}, $\mc{P}$ yields smaller average delay.

If, in addition, $\text{MEAN}(\mc{P})=\text{MEAN}(\mc{P}')$. Then by letting $n=\bar{\tau}+1$, $x_i=i-1$, $\phi(x)=x^2$,  $\alpha_i=P_{i-1}$, and $\beta_i=P_{i-1}'$ in Theorem 1 in \cite{hadar1969rules}, we have
\begin{equation}\label{eq:hadarvariance}
\begin{split}
    \sum_{j=0}^{\bar{\tau}} j^2P_j\le \sum_{j=0}^{\bar{\tau}} j^2P_j'.
\end{split}
\end{equation}
In addition, $\text{MEAN}(\mc{P})=\text{MEAN}(\mc{P}')$, which, together with \eqref{eq:hadarvariance}, yields
\begin{equation*}
\begin{split}
    \text{VAR}(\mc{P})&=\sum_{j=0}^{\bar{\tau}} j^2P_j-(\text{MEAN}(\mc{P}))^2\\
    &\le \sum_{j=0}^{\bar{\tau}} j^2P_j'-(\text{MEAN}(\mc{P}'))^2\\
    &=\text{VAR}(\mc{P}'),
\end{split}
\end{equation*}
\emph{i.e.}, $\mc{P}$ yields smaller variance.

\section{Proof of Lemma \ref{lemma:stochasticdominance}}\label{append:proofoflemmastochasticdominance}

When the delay distribution is specialized to $\mc{P}$ and $\mc{P}'$, the function $\phi$ defined in Theorem \ref{thm:adaptbounded} becomes 
\begin{align*}
    &\text{Distribution } \mc{P}:~~\phi(\rho):=\rho-\rho_c-\frac{c^2}{m}(\sum_{i=0}^{\bar{\tau}}P_i\rho^{-i}-1),\\
    &\text{Distribution } \mc{P}':~~\phi'(\rho):=\rho-\rho_c-\frac{c^2}{m}(\sum_{i=0}^{\bar{\tau}}P_i'\rho^{-i}-1).
\end{align*}

To prove the result, we first show that if $\mc{P}\succeq_1 \mc{P}'$, then for any $\rho\in(0,1)$,
\begin{equation}\label{eq:phirhocomparison}
    \phi(\rho)\ge \phi'(\rho).
\end{equation}
By the definitions of $\phi(\rho)$ and $\phi'(\rho)$, \eqref{eq:phirhocomparison} holds if
\begin{equation}\label{eq:rhominusi}
    \sum_{i=0}^{\bar{\tau}}P_i \rho^{-i}\le \sum_{i=0}^{\bar{\tau}}P_i' \rho^{-i}.
\end{equation}
Letting $n=\bar{\tau}+1$, $x_i=i-1$, $\phi(x)=\rho^{-x}$,  $\alpha_i=P_{i-1}$, and $\beta_i=P_{i-1}'$ in Theorem 1 in \cite{hadar1969rules}, we have that if $\mc{P}\succeq_1\mc{P}'$, then \eqref{eq:rhominusi} holds, which further guarantees \eqref{eq:phirhocomparison}.

By \eqref{eq:phirhocomparison} and \eqref{eq:phirhopositivenegative}, we have
\begin{align*}
    &\phi(\rho_a)\overset{\eqref{eq:phirhocomparison}}{\ge} \phi'(\rho_a)\overset{\eqref{eq:phirhopositivenegative}}{\ge} 0,\\
    &\phi'(\rho_c)\overset{\eqref{eq:phirhocomparison}}{\le} \phi(\rho_c)\overset{\eqref{eq:phirhopositivenegative}}{<}0,
\end{align*}
which implies 
\begin{equation*}
    -\phi(\rho_a)/\phi(\rho_c)\ge -\phi'(\rho_a)/\phi'(\rho_c).
\end{equation*}
Then by \eqref{eq:alphavalue}, $\alpha_{\mc{P}}\le \alpha_{\mc{P}'}$. Moreover, $\rho_a\ge \rho_c$, $\rho_{\mc{P}}=\alpha_{\mc{P}}\rho_a+(1-\alpha_{\mc{P}})\rho_c$, and $\rho_{\mc{P}'}=\alpha_{\mc{P}'}\rho_a+(1-\alpha_{\mc{P}'})\rho_c$. Therefore, $\rho_{\mc{P}}\le \rho_{\mc{P}'}$. Completes the proof.

\section{Proof of Theorem \ref{thm:general}}\label{appendix:proofthmgeneral}


We first define all the notations that will be used in the proof. We define the index sequence $\{\mc{I}(t)\}_{t\in\N_0}$ as: $\mc{I}(0)=0$ and for each $t\in\N$,
\begin{equation}\label{eq:defIt}
    \mathcal{I}(t) = \min\{k': k-\tau(k)\ge \mathcal{I}(t-1)~\text{for all }k\ge k'\}+1,
\end{equation}
and let
\begin{equation*}
    a(t):=\operatorname{\arg\;\max}_{k\in [\mathcal{I}(t), \mathcal{I}(t+1))} V(k)
\end{equation*}
Moreover, we use the same definitions of $V(k)$ and $W(k)$ as in Appendix \ref{appendix:proofthmboundegeneral}: \begin{align*}
    V(k)&=\mathbb{E}[\|\bx(k)-\bx^\star\|^2],\\
    W(k)&=\frac{1-\alpha}{m\alpha}\mathbb{E}[\|(\operatorname{Id}-\T)(\bx(k))\|^2].
\end{align*}
To understand the sequence $\mc{I}(t)$, note that for each $t\in \mathbb{N}$, by the definition in \eqref{eq:defIt},
\begin{equation}\label{eq:kminustauk}
    k-1-\tau(k-1)\ge \mathcal{I}(t-1),~\forall k\ge \mathcal{I}(t).
\end{equation}
Hence, $\{\mathcal{I}(t)\}$ defines the following \emph{Markov} property for the update \eqref{eq:Alg1formal}: \emph{For each $t\in\mathbb{N}$, all the iterates $\bx(k)$, $k\ge \mathcal{I}(t)$ are determined by $\bx(k)$, $k\in [\mathcal{I}(t-1), \mathcal{I}(t))$ and do not rely on earlier iterates.}
Under Assumption \ref{asm:parallelasynchrony}, the sequence $\{\mathcal{I}(t)\}$ is well defined because given $\mc{I}(t)$, there always exists $k'$ such that $k-\tau(k)\ge \mc{I}(t)$ for all $k\ge k'$.

The remaining proof includes three steps. Step 1 derives that for any $t\in\N$ and $k\in [\mathcal{I}(t), \mathcal{I}(t+1))$,
\begin{equation}\label{eq:Vmaxkmonotone}
    V(k)\le \max_{\ell\in [\mathcal{I}(t-1), \mathcal{I}(t))} V(\ell).
\end{equation}
Base on step 1, step 2 proves
\begin{equation}\label{eq:successivevatotal}
\begin{split}
    V(a(t))\le V(a(t-1))- W(a(t)-1-\tau(a(t)-1)),
\end{split}
\end{equation}
which is further used in step 3 to show the result.

\textbf{Step 1}: By \eqref{eq:oldboundedLyapunovfirm},
\begin{equation}\label{eq:Lyapunovfirm}
	\begin{split}
		V(k+1) \le& \max_{k-\tau(k)\le t\le k} V(t) - W(k-\tau(k)).
	\end{split}
\end{equation}
With \eqref{eq:Lyapunovfirm}, one can prove \eqref{eq:Vmaxkmonotone} by induction. To this end, note from \eqref{eq:kminustauk} that $k-1-\tau(k-1)\ge \mc{I}(t-1)$ for any $k\in [\mathcal{I}(t), \mathcal{I}(t+1))$. Then, by \eqref{eq:Lyapunovfirm} with $k=\mc{I}(t)-1$, the equation \eqref{eq:Vmaxkmonotone} holds naturally for $k=\mathcal{I}(t)$. Suppose that \eqref{eq:Vmaxkmonotone} holds for all $k\in [\mc{I}(t), k']$ for some $k'\in [\mathcal{I}(t), \mathcal{I}(t+1)-1)$. Then, by \eqref{eq:kminustauk} we have $k'-\tau(k')\in [\mathcal{I}(t-1), k']$, which, together with \eqref{eq:Lyapunovfirm} at $k=k'$, yields \eqref{eq:Vmaxkmonotone} at $k=k'+1$. Following this induction procedure we derived \eqref{eq:Vmaxkmonotone} for all $k\in [\mathcal{I}(t), \mathcal{I}(t+1))$.

\textbf{Step 2}: By \eqref{eq:kminustauk},
\begin{align*}
    & a(t)-1-\tau(a(t)-1) \in [\mc{I}(t-1), \mc{I}(t+1)),\\
    & a(t)-1\in [\mc{I}(t-1), \mc{I}(t+1))
\end{align*}
and therefore
\begin{equation*}
    [a(t)-1-\tau(a(t)-1), a(t)-1]\subseteq [\mc{I}(t-1), \mc{I}(t+1)).
\end{equation*}
Substituting the above equation into \eqref{eq:Lyapunovfirm} at $k=a(t)-1$ gives
\begin{equation}\label{eq:Vat}
        V(a(t))\le \max_{\ell\in [\mc{I}(t-1), \mc{I}(t+1))}V(\ell)- W(a(t)-1-\tau(a(t)-1)).
\end{equation}
In addition, by maximizing the left-hand side of \eqref{eq:Vmaxkmonotone} over $k\in [\mc{I}(t-1), \mc{I}(t+1))$, we have
\begin{equation}\label{eq:Itminus1Itplus1}
    \max_{\ell\in [\mc{I}(t-1), \mc{I}(t+1))}V(\ell)\le \max_{\ell\in [\mathcal{I}(t-1), \mathcal{I}(t))} V(\ell)
\end{equation}
and therefore,
\begin{equation}\label{eq:maxvequatovat1}
\begin{split}
        \max_{\ell\in [\mc{I}(t-1), \mc{I}(t+1))}V(\ell) &= \max_{\ell\in [\mathcal{I}(t-1), \mathcal{I}(t))} V(\ell)\\
        &= V(a(t-1)).
\end{split}
\end{equation}
Substituting \eqref{eq:maxvequatovat1} into \eqref{eq:Vat}, we have \eqref{eq:successivevatotal}.

\textbf{Step 3}: Adding \eqref{eq:successivevatotal} from $t=1$ to $t=+\infty$ gives
\begin{equation*}
    \begin{split}
        \sum_{t=1}^\infty W(a(t)-1-\tau(a(t)-1))&\le V(a(0))=V(0),
    \end{split}
\end{equation*}
where $V(a(0))=V(0)$ can be easily derived from \eqref{eq:boundedLyapunovfirm}. Therefore, $\lim_{t\rightarrow +\infty} W(a(t)-1-\tau(a(t)-1))=0$, which completes the proof.

\section{Proof of Theorem \ref{thm:unbounded}}\label{sec:convanasublineargrow}

We use the same definitions of $V(k)$, $W(k)$, and $\mc{I}(t)$ as in Appendix \ref{appendix:proofthmgeneral} and let $q=1-\frac{1-c^2}{m}$. We first show that for any $t\in\N_0$ and $k\in [\mc{I}(t), \mc{I}(t+1))$,
\begin{equation}\label{eq:Vsublinearuboundedproof}
\begin{split}
        &V(k)\le q^tV(0).
\end{split}
\end{equation}
Subsequently, we prove that for any $k\in [\mc{I}(t), \mc{I}(t+1))$,
\begin{equation}\label{eq:tboundbyk}
        t\ge \begin{cases}
            \Theta(k^{1-\beta}), & \beta\in(0,1),\\
            \Theta(\ln k) & \beta=1.
        \end{cases}
\end{equation}
Combining the above two equations yields the result.

\textbf{Proof of \eqref{eq:Vsublinearuboundedproof}}: Fix $k\in [\mathcal{I}(t)-1, \mathcal{I}(t+1)-1)$. We have by \eqref{eq:kminustauk} that $k-\tau(k) \in [\mathcal{I}(t-1), \mathcal{I}(t+1)-1)$. 
Moreover, since any pseudo-contractive operator is also averaged, \eqref{eq:Itminus1Itplus1} holds, so that
\begin{align*}
   V(k)&\le \max_{\ell\in[\mc{I}(t-1), \mc{I}(t))}V(\ell),\\
   V(k-\tau(k))&\le \max_{\ell\in[\mc{I}(t-1), \mc{I}(t))}V(\ell),
\end{align*}
which, together with \eqref{eq:boundedcontractionLyap}, yields
\begin{equation*}
    V(k+1)\le q\max_{\ell\in [\mathcal{I}(t-1), \mathcal{I}(t))} V(\ell).
\end{equation*}
Maximizing the left-hand side of the above equation over $k\in [\mathcal{I}(t)-1, \mathcal{I}(t+1)-1)$, we obtain
\begin{equation*}
\begin{split}
        \max_{\ell\in [\mathcal{I}(t), \mathcal{I}(t+1))} V(\ell)&\le q\max_{\ell\in [\mathcal{I}(t-1), \mathcal{I}(t))} V(\ell)\\
    &\le q^t\max_{\ell\in [\mathcal{I}(0), \mathcal{I}(1))}V(\ell)\\
    &=q^tV(0),
\end{split}
\end{equation*}
where the last step uses $V(k)\le V(0)$ $\forall k\in\N_0$ derived from \eqref{eq:Lyapunovfirm}. Therefore, \eqref{eq:Vsublinearuboundedproof} holds.

\textbf{Proof of \eqref{eq:tboundbyk}}:
We prove \eqref{eq:tboundbyk} by showing that for any $k\in [\mc{I}(t), \mc{I}(t+1))$,
\begin{equation}\label{eq:tlowerbound}
        t\ge \begin{cases}
            \frac{1}{a}(k+\gamma+1)^{1-\beta}-\gamma-1, & \beta\in(0,1),\\
            \frac{\ln(\eta(k+1)/(\gamma+1))}{\ln (1/(1-\eta))} & \beta=1,
        \end{cases}
\end{equation}
where $a=\eta(1-\eta)^{-\beta}+\gamma+1$. We consider two cases of $\beta$ separately.
    
    \emph{Case 1: $\beta\in(0,1)$.} By the definition of $\{\mc{I}(t)\}$ in \eqref{eq:defIt}, for each $t\in\N_0$,
    \begin{equation}\label{eq:Itsequence}
        \mc{I}(t+1) - 2 - \tau(\mc{I}(t+1) - 2)\le \mc{I}(t)-1.
    \end{equation}
    Moreover, by Assumption \ref{asm:sublineardelay},
    \begin{equation*}
    \begin{split}
        \tau(\mc{I}(t+1) - 2)&\le \eta(\mc{I}(t+1) - 2)^\beta+\gamma\\
        &\le \eta(\mc{I}(t+1) - 2)+\gamma,
    \end{split}
    \end{equation*}
    which, together with \eqref{eq:Itsequence}, yields
    \begin{equation}\label{eq:Itplus1bound}
        \mc{I}(t+1)-2\le \frac{\mc{I}(t)+\gamma-1}{1-\eta}.
    \end{equation}
    Then,
    \begin{equation*}
    \begin{split}
        \tau(\mc{I}(t+1) - 2)& \le \eta(\mc{I}(t+1) - 2)^\beta+\gamma\\
        &\le \eta\left(\frac{\mc{I}(t)+\gamma-1}{1-\eta}\right)^\beta+\gamma.
    \end{split}
    \end{equation*}
    Substituting the above equation into \eqref{eq:Itsequence} ensures
    \begin{equation}\label{eq:It1smallthanIt}
        \begin{split}
            \mc{I}(t+1) &\le \mc{I}(t)+\eta\left(\frac{\mc{I}(t)+\gamma-1}{1-\eta}\right)^\beta+\gamma+1\\
            \le& \mc{I}(t)+(\eta(1-\eta)^{-\beta}+\gamma+1)(\mc{I}(t)+\gamma)^\beta\\
            =& \mc{I}(t)+a(\mc{I}(t)+\gamma)^\beta,
        \end{split}
    \end{equation}
    where the second step uses $\gamma+1\le (\gamma+1)\gamma^\beta\le (\gamma+1)(\mc{I}(t)+\gamma)^\beta$derived from $\gamma\ge 1$.
    Based on \eqref{eq:It1smallthanIt}, by induction we will show
    \begin{equation}\label{eq:Itupperbound}
        \mc{I}(t) \le (a(t+\gamma))^{\frac{1}{1-\beta}}-\gamma,~\forall t\in\N_0.
    \end{equation}
    When $t=0$, \eqref{eq:Itupperbound} holds because $\gamma\ge 1$, $a\ge 1$, and $\beta\in(0,1)$. Suppose \eqref{eq:Itupperbound} holds for some $t\in\N_0$. Then, by \eqref{eq:It1smallthanIt} and \eqref{eq:Itupperbound} we have
    \begin{equation}
        \begin{split}
            \mc{I}(t+1)&\le (a(t+\gamma))^{\frac{1}{1-\beta}}-\gamma+a(a(t+\gamma))^{\frac{\beta}{1-\beta}}\\
            &=(a(t+\gamma))^{\frac{\beta}{1-\beta}}(a(t+\gamma+1))-\gamma\\
            &\le (a(t+\gamma+1))^{\frac{1}{1-\beta}}-\gamma,
        \end{split}
    \end{equation}
    \emph{i.e.}, \eqref{eq:Itupperbound} holds for $t+1$. Following the induction procedure we have that \eqref{eq:Itupperbound} holds for all $t\in\N_0$. Then, for any $k\in [\mc{I}(t), \mc{I}(t+1))$, we have
    \begin{equation*}
        k\le \mc{I}(t+1)-1 \le (a(t+\gamma+1))^{\frac{1}{1-\beta}}-\gamma-1,
    \end{equation*}
    which further implies \eqref{eq:tlowerbound} with $\beta\in(0,1)$.

    \emph{Case 2: $\beta=1$.} Adding $\frac{\gamma+1}{\eta}$ to both sides of \eqref{eq:Itplus1bound} ensures
    \begin{equation*}
    \begin{split}
        \mc{I}(t+1)+\frac{\gamma+1}{\eta}-2&\le \frac{\mc{I}(t)+\frac{\gamma+1}{\eta}-2}{1-\eta},
    \end{split}
    \end{equation*}
    which further yields
    \begin{equation*}
    \begin{split}
        \mc{I}(t+1)+\frac{\gamma+1}{\eta}-2&\le (1-\eta)^{-(t+1)}(\mc{I}(0)+\frac{\gamma+1}{\eta}-2)\\
        &= (\frac{\gamma+1}{\eta}-2)(1-\eta)^{-(t+1)}\\
        &\le \frac{\gamma+1}{\eta}(1-\eta)^{-(t+1)}.
    \end{split}
    \end{equation*}
    Then, for any $k\in [\mc{I}(t), \mc{I}(t+1))$,
    \begin{equation*}
        k\le \mc{I}(t+1)-1\le \frac{\gamma+1}{\eta}(1-\eta)^{-(t+1)}-1,
    \end{equation*}
    which further gives \eqref{eq:tlowerbound} with $\beta=1$. This completes the proof.

\section{Equivalence to ADMM} \label{sec:ADMMequivalence}
Running \eqref{eq:TiADMMz}--\eqref{eq:TiADMMx} with $\lambda=1$ synchronously and indexing the iterate by $k\in\N_0$ ensures
\begin{align}
    & z(k) = \frac{1}{m}\sum_{i=1}^m x_i(k),\label{eq:ourspecializationtoADMMz}\\
    & y_i(k+1) = \prox_{\gamma F_i} (2z(k)-x_i(k)),~\forall i\in[m],\\
    & x_i(k+1) = x_i(k)+y_i(k+1)-z(k),~\forall i\in[m],\label{eq:ourspecializationtoADMMx}
\end{align}
where $z_i$ in \eqref{eq:TiADMMz} is indexed by $z(k)$. ADMM for solving problem \eqref{eq:consensusprob} takes the following form \cite{zhang2014asynchronous}:
\begin{align}
    \hat{y}_i(k+1) &=  \prox_{F_i/\eta}(\hat{z}(k)-\hat{x}_i(k)/\eta),~\forall i\in [m],\label{eq:ADMMtousy}\\
    \hat{z}(k+1) &= \frac{1}{m}\sum_{i=1}^m (\hat{y}_i(k+1)+\hat{x}_i(k)/\eta),\\
    \hat{x}_i(k+1) &= \hat{x}_i(k)+\eta(\hat{y}_i(k+1)-\hat{z}(k+1)),~\forall i\in [m].\label{eq:ADMMtousx}
\end{align}
It can be verified that if $x_i(k)$, $y_i(k)$, $z_i(k)$ update according to \eqref{eq:ourspecializationtoADMMz}--\eqref{eq:ourspecializationtoADMMx}. Then, by letting $\eta = 1/\gamma$, $\hat{x}_i(k)=\eta(x_i(k)-z(k))$, $\hat{y}_i(k)=y_i(k)$, and $\hat{z}(k)=z(k)$, the updates of $\hat{x}_i(k)$, $\hat{y}_i(k)$, $\hat{z}(k)$ follow \eqref{eq:ADMMtousy}--\eqref{eq:ADMMtousx}. One important step of the verification is $\frac{1}{m}\sum_{i=1}^m \hat{x}_i(k) = \frac{\eta}{m}\sum_{i=1}^m (x_i(k) - z(k)) = \mathbf{0}$, which yields $\hat{z}(k+1) = \frac{1}{m}\sum_{i=1}^m \hat{y}_i(k+1)$.

\section{Proof of Lemma \ref{lemma:ADMMlemma}}\label{ssec:proofofADMM}
By Proposition 12.27 in \cite{bauschke2011convex}, both $\prox_{\gamma F}$ and $\prox_{\gamma \mc{I}_{\mc{C}}}$ are $1/2$-averaged (equivalent to firm non-expansiveness in Proposition 12.27 in \cite{bauschke2011convex}). Then, using Proposition 4.21 in \cite{bauschke2011convex}, the operator $\T'=\prox_{\gamma F}\circ(2\prox_{\gamma \mc{I}_{\mc{C}}}-\Id)-\prox_{\mc{I}_{\mc{C}}}+\Id$ is $1/2$-averaged. Hence, $\T'=\frac{R}{2}+\frac{\Id}{2}$ for some non-expansive operator $R$, so that
\begin{equation*}
    \T=(1-\lambda)\Id+\lambda \T'=(1-\lambda/2)\Id+(\lambda/2)R.
\end{equation*}
Therefore, $\T$ is $\lambda/2$-averaged.

\section{Proof of Lemma \ref{lemma:extendedADMM}}\label{sec:proofofextendedADMM}

By the proof of Theorem 25.8 in \cite{bauschke2011convex}, the proximal gradient operator $\prox_{\gamma r}\circ (\Id-\gamma\nabla f)$ is $1/(2\theta)$-averaged where $\theta=\frac{\min\{1,L/\gamma\}+1/2}{2}$. Therefore, there exists a non-expansive operator $R$ satisfying $\prox_{\gamma r}\circ (\Id-\gamma\nabla f) = R/(2\theta)+(1-1/(2\theta))\Id$, so that
\begin{equation*}
\begin{split}
    \T' &= \theta\prox_{\gamma r}\circ (\Id-\gamma\nabla f)+(1-\theta)\Id\\
    &= R/2+\Id/2.
\end{split}
\end{equation*}
Therefore, $\T'$ is $1/2$-averaged.

Since $\T'$ and $\prox_{\gamma\mc{I}_{\mc{Z}}}$ are $1/2$-averaged, by using Proposition 4.21 in \cite{bauschke2011convex}, the operator $\T'\circ(2\prox_{\gamma \mc{I}_{\mc{Z}}}-\Id)-\prox_{\mc{I}_{\mc{Z}}}+\Id$ is $1/2$-averaged. Then, $\T'\circ(2\prox_{\gamma \mc{I}_{\mc{Z}}}-\Id)-\prox_{\mc{I}_{\mc{Z}}}+\Id=\frac{R}{2}+\frac{\Id}{2}$ for some non-expansive operator $R$, so that
\begin{equation*}
\begin{split}
    \T&=(1-\lambda)\Id+\lambda (\T'\circ(2\prox_{\gamma \mc{I}_{\mc{Z}}}-\Id)-\prox_{\mc{I}_{\mc{Z}}}+\Id)\\
    &=(1-\lambda/2)\Id+(\lambda/2)R.
\end{split}
\end{equation*}
Therefore, $\T$ is $\lambda/2$-averaged.



\section{Delay distribution}\label{sec:delaydistribution}

To validate the phenomenon that most delays may be much smaller than the maximum delay, we plot the delay distribution generated by the experiments (theoretical parameter) in Section \ref{sec:experiment}. Specifically, for both Lasso (Figure \ref{fig:LASSOCIFAR}) and Logistic regression (Figure \ref{fig:LogisticCIFAR}), we plot the delay distribution generated by all four asynchronous algorithms.

	\begin{figure*}[!htb]
		\centering
		\caption{Delay distribution (x-axis represents delay)}
		\subfigure[Lasso]{
			\includegraphics[scale=0.75]{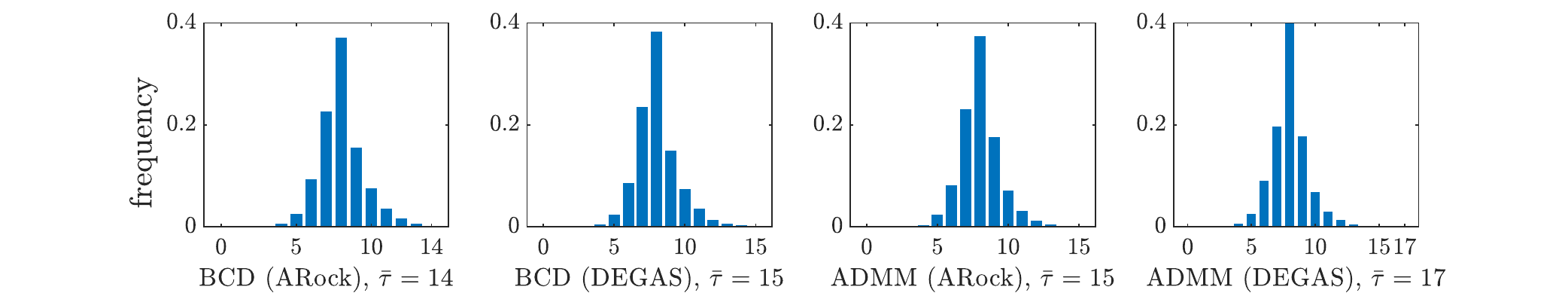}\label{fig:delaylassocifar_iter}}
		\subfigure[Logistic]{
			\includegraphics[scale=0.75]{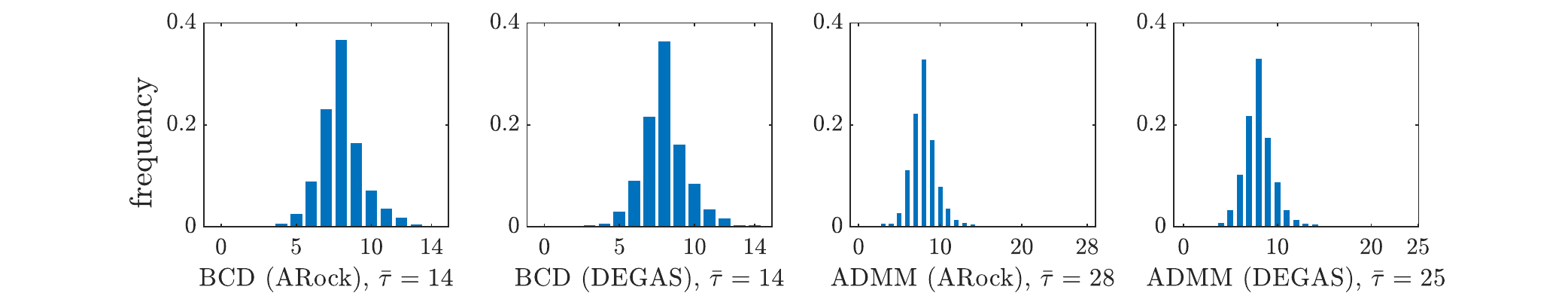}\label{fig:delaylassocifar_time}}
		\label{fig:delayLASSOCIFAR}
	\end{figure*}

Observe from Figure \ref{fig:delayLASSOCIFAR} that most delays are much smaller than the maximum delay. For example, in Figure \ref{fig:delayLASSOCIFAR} (a) BCD (ARock), the maximum delay is $14$ and over $86\%$ delays are smaller than or equal to $9$, and in Figure \ref{fig:delayLASSOCIFAR} (b) ADMM (ARock), the maximum delay is $28$ while over $97\%$ delays are smaller than or equal to $11$.







\section{Comparison in terms of running time}\label{sec:exprunningtime}

We conducted a comparison between DEGAS, ARock, and their common synchronous algorithms based on their running time. We consider two scenarios: no straggler and with straggler. In the first setting, we use the same experiment setting as in Section \ref{sec:experiment}, where each worker is a core in a 10-core machine and all the workers are homogeneous. In the second setting, we choose one worker as the straggler and let it sleep for twice its local computation time at each iteration. The "sleep" scheme for setting a straggler is standard in the literature \cite{lian2018asynchronous,luo2020prague}. However, existing works usually choose a random straggler at each iteration, while we consider the more practical setting where the straggler is fixed. We use the theoretical step-size setting in Section \ref{sec:experiment}.

The experiment results are presented in Figures \ref{fig:stragglerLasso}--\ref{fig:stragglerLogistic}. Observe from Figures \ref{fig:stragglerLasso}--\ref{fig:stragglerLogistic} that DEGAS is slower than the synchronous methods if there is no straggler, and is faster in the case of including straggler. This phenomenon is reasonable. First, without a straggler, the numbers of $\T_i$ computed by the synchronous and asynchronous methods are very close according to our observation, while the performance of asynchronous methods is degraded by the information delay. Second, when a straggler is present, the average per-iteration time consumption of the synchronous methods significantly increases, while that of the asynchronous methods only increases slightly.

\begin{figure*}[htb]
    \centering
    \caption{Convergence for solving Lasso: no straggler v.s. with straggler}
    \subfigure[no straggler]{
	\includegraphics[scale=0.4]{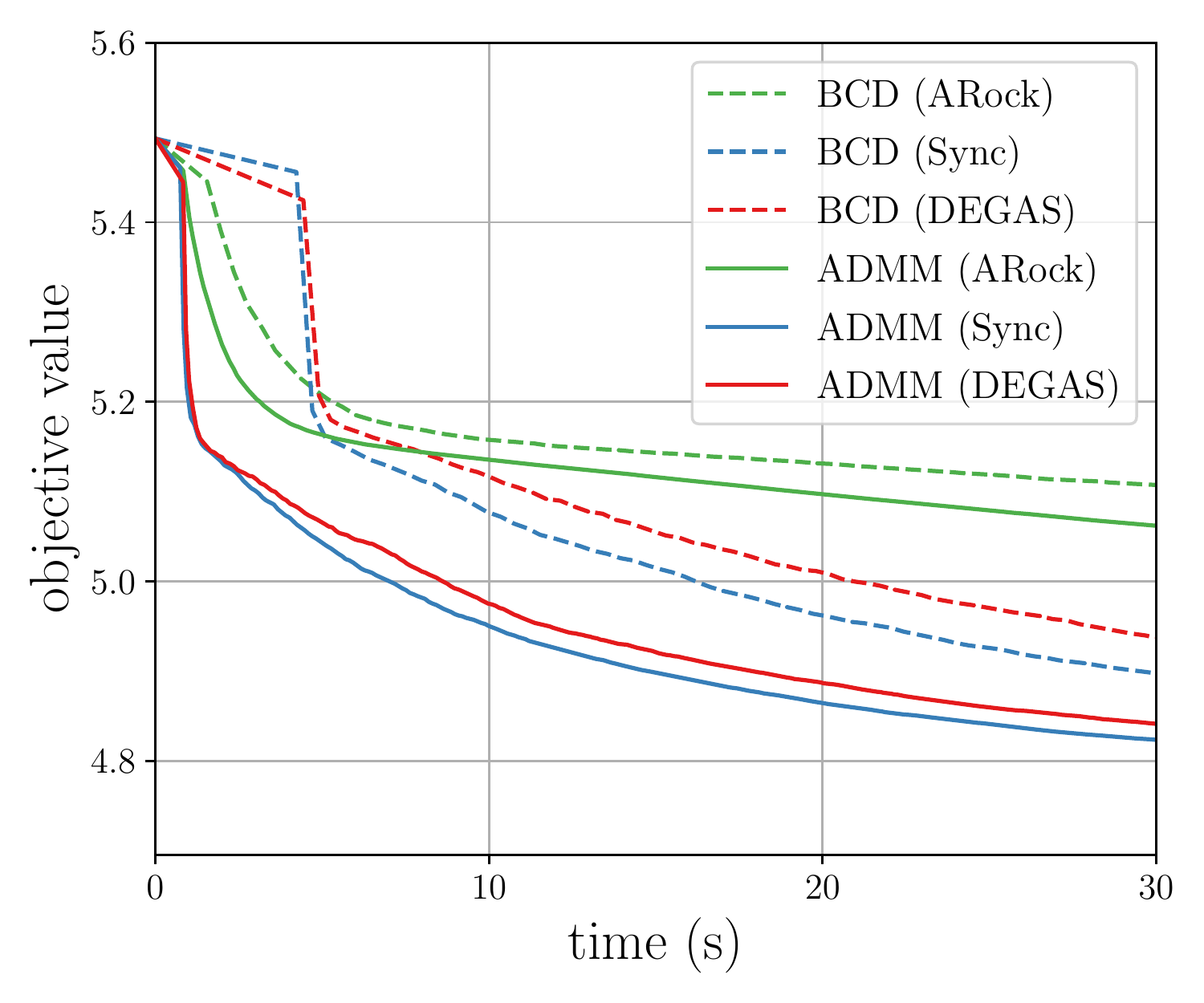}\label{fig:lasso_cifar_time}}
    \subfigure[with straggler]{	\includegraphics[scale=0.4]{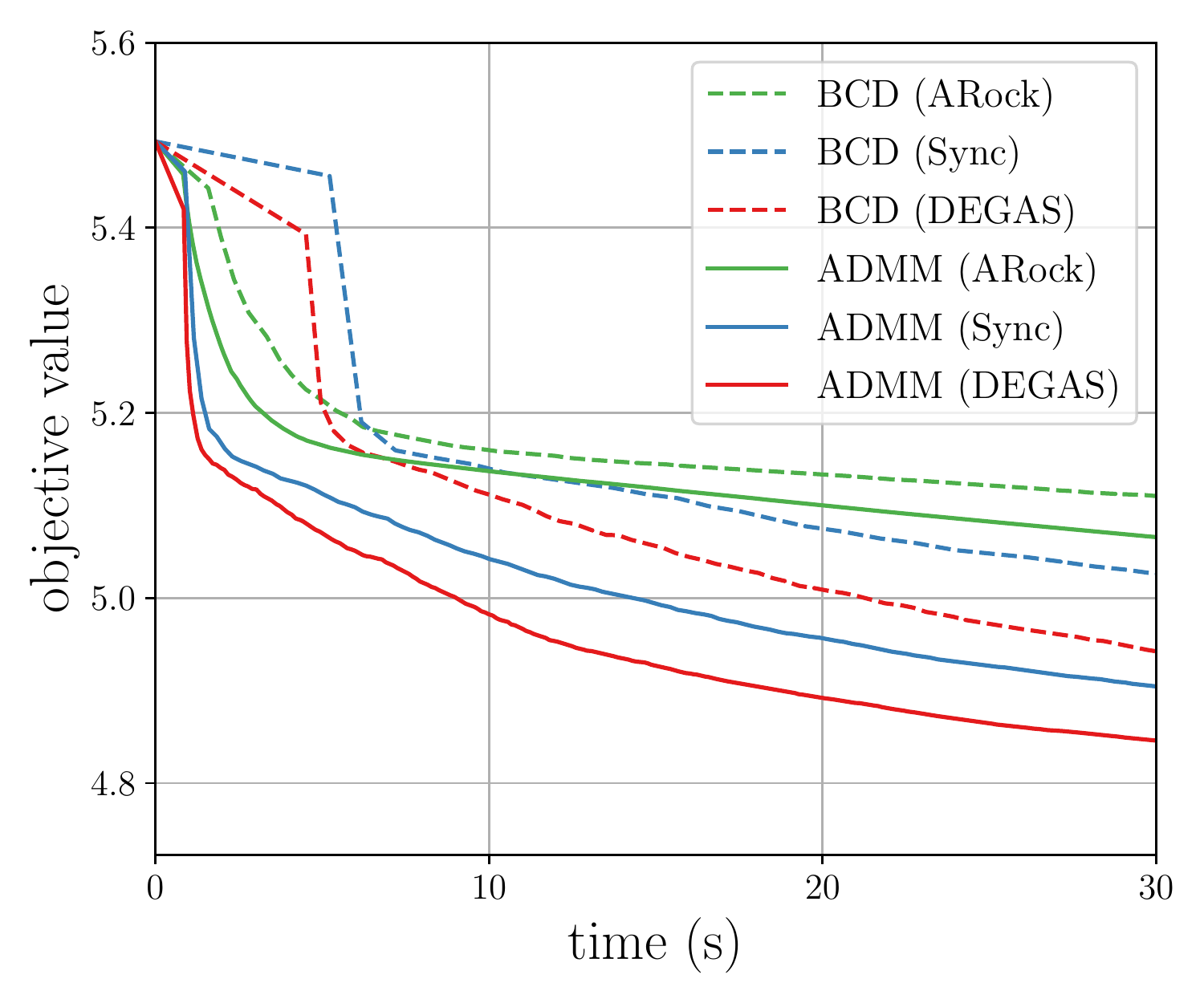}\label{fig:lassocifar_straggler_time}}
    \label{fig:stragglerLasso}
\end{figure*}

\begin{figure*}[htb]
    \centering
    \caption{Convergence for solving Logistic regression: no straggler v.s. with straggler}
    \subfigure[no straggler]{
	\includegraphics[scale=0.4]{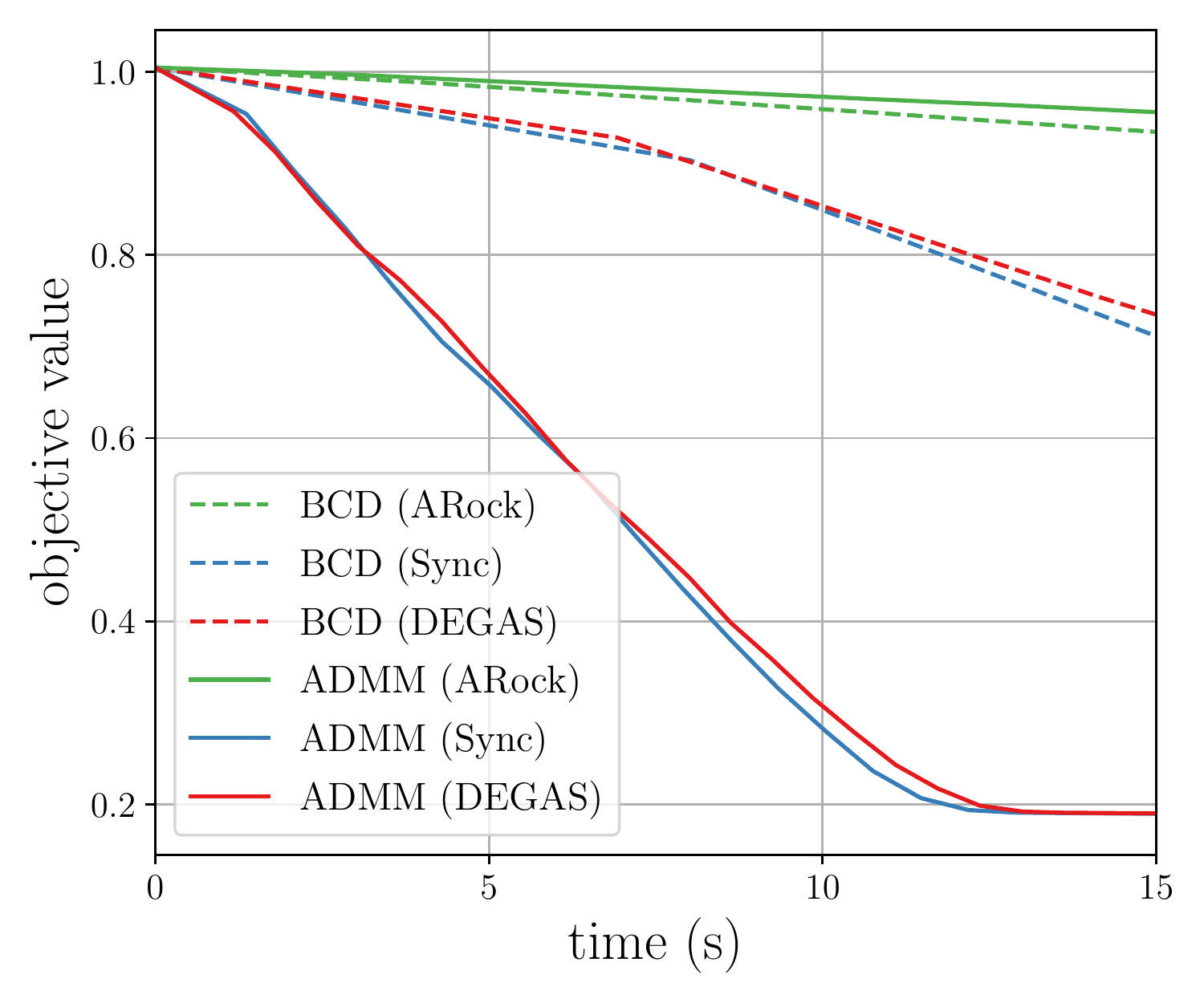}\label{fig:logisticrcv_time}}
    \subfigure[with straggler]{
    \includegraphics[scale=0.4]{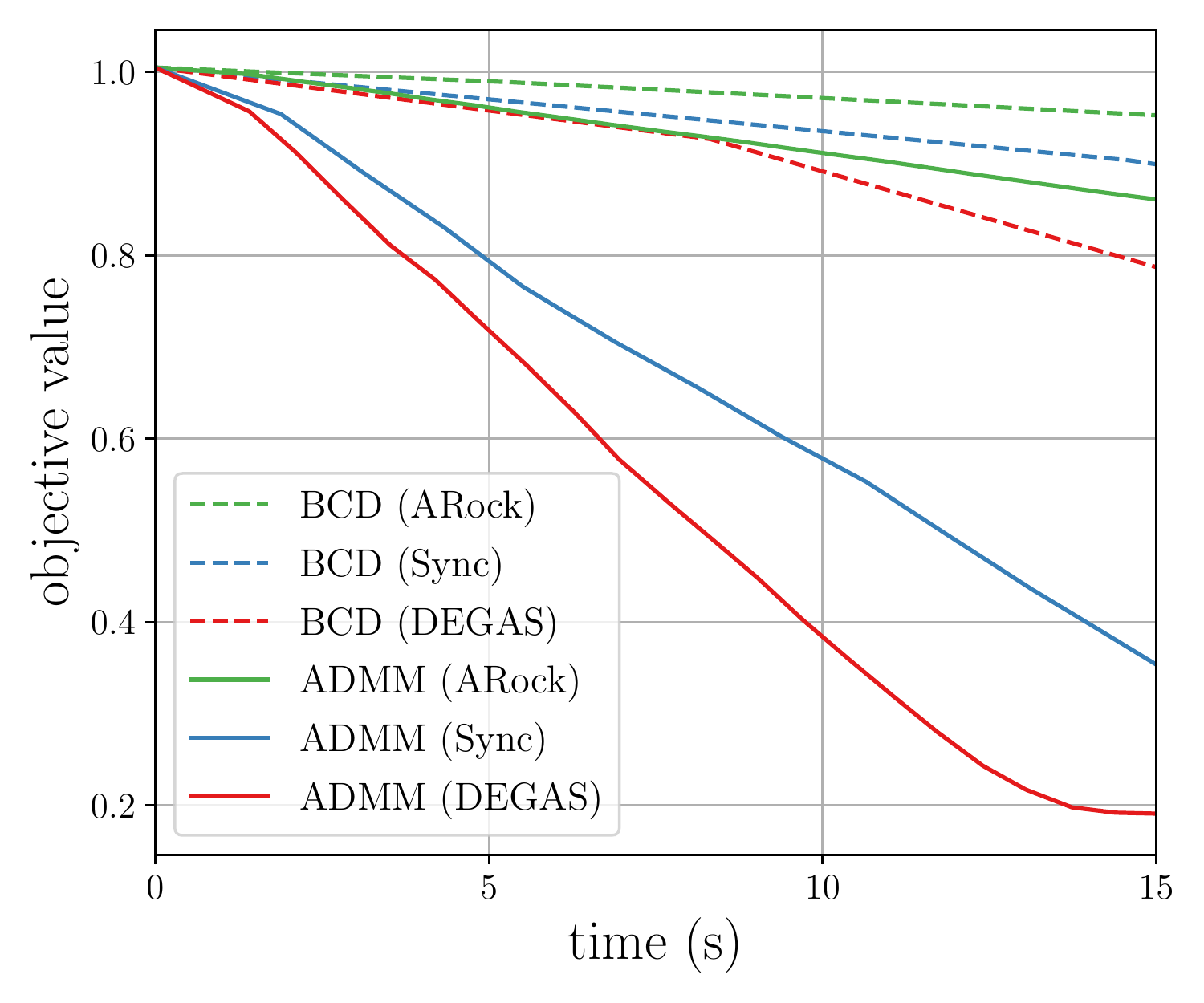}\label{fig:logisticrcv_iter}}
    \label{fig:stragglerLogistic}
\end{figure*}

\end{document}